\documentclass[11pt]{article}
\usepackage[T1]{fontenc}
\usepackage[utf8]{inputenc}
\usepackage{lmodern}
\usepackage{amsmath,amssymb,amsthm,mathrsfs}
\usepackage{mathtools}
\usepackage[margin=1in]{geometry}
\usepackage{microtype}
\usepackage{parskip}
\usepackage{needspace}
\usepackage{xcolor}
\usepackage{enumitem}
\usepackage{booktabs,graphicx,tikz,placeins}
\usetikzlibrary{arrows.meta,calc,patterns}
\usepackage[hidelinks,bookmarksnumbered=true]{hyperref}

\hypersetup{pdftitle={Hausdorff dimension and Lebesgue measure of rotated planar difference sets},pdfauthor={Thang Pham, Chun-Yen Shen and Dung The Tran},pdfsubject={Rotated planar difference sets}}
\setlist{nosep}
\numberwithin{equation}{section}
\newtheorem{theorem}{Theorem}[section]
\newtheorem{proposition}[theorem]{Proposition}
\newtheorem{lemma}[theorem]{Lemma}
\newtheorem{corollary}[theorem]{Corollary}
\theoremstyle{definition}
\newtheorem{remark}[theorem]{Remark}
\newcommand{\R}{\mathbb R}
\newcommand{\HH}{\mathcal H}
\newcommand{\LL}{\mathcal L}
\newcommand{\OO}{\mathrm O(2)}
\newcommand{\SO}{\mathrm{SO}(2)}
\newcommand{\dimH}{\dim_{\mathrm H}}
\newcommand{\supp}{\operatorname{supp}}
\newcommand{\one}{\mathbf 1}
\newcommand{\cD}{\mathcal D}
\newcommand{\cT}{\mathcal T}
\newcommand{\dd}{\,d}
\newcommand{\norm}[1]{\left\|#1\right\|}
\newcommand{\less}{\lesssim}

\newcommand{\sbase}{s_0}
\newcommand{\dimP}{\dim_{\mathrm P}}
\DeclareMathOperator{\dist}{dist}

\title{On Mattila’s intersection problem in the plane: sharp dimension and positive measure of rotated difference sets}
\author{%
  Thang Pham\thanks{Institute of Mathematics and Interdisciplinary Sciences,
    Xidian University. \newline
\hspace*{0.45cm} Email: \texttt{thangphammath@xidian.edu.cn}.}%
\and Chun-Yen Shen \thanks{National Taiwan University. ~Email: {\tt cyshen@math.ntu.edu.tw.}}
  \and Dung The Tran \thanks{VNU University of Science, Hanoi, Vietnam. ~Email: {\tt tranthedung56@gmail.com.}}
}
\date{}
\begin{document}
\maketitle

\begin{abstract}
Let \(A,B\subset\R^2\) be nonempty Borel sets. We prove that
\[
 \dimH(gA-B)\ge\min\{\dimH A+\dimH B,2\}
\]
for almost all \(g\in\mathrm O(2)\). This bound is sharp in general,
and the proof gives a corresponding statement for the lower Hausdorff
dimension of convolutions of Frostman measures. We also prove that
\(\LL^2(gA-B)>0\) for almost all \(g\) if both dimensions exceed
\((\sqrt{13}-1)/2\). The second result follows from an angular
decomposition, refined decoupling and a two-ends Furstenberg estimate. Some applications on intersection patterns are also addressed.
\end{abstract}
{\footnotesize
\noindent\textit{2020 Mathematics Subject Classification.} Primary 28A78; Secondary 42B37.\par
\noindent\textit{Key words and phrases.} Hausdorff dimension, Lebesgue measure, difference set, projection, refined decoupling.\par}


\section{Introduction}
\label{sec:introduction}

The size of the intersection of two sets after a rigid motion is a
classical question in geometric measure theory. For nonempty Borel sets
\(A,B\subset\R^2\) and an orthogonal transformation \(g\in\OO\), set
\[
 gA-B=\{ga-b:a\in A,\ b\in B\}.
\]
This is the set of translations for which the two sets meet, since
\[
 z\in gA-B\quad\Longleftrightarrow\quad gA\cap(B+z)\ne\varnothing.
\]
We study its Hausdorff dimension and Lebesgue measure. We write
\(\dimH\) for Hausdorff dimension, \(\dimP\) for packing dimension,
\(\HH^s\) for \(s\)-dimensional Hausdorff measure, and \(\LL^n\) for
Lebesgue measure on \(\R^n\). The Haar probability measure on \(\OO\)
is denoted by \(m\). The phrase ``almost all orthogonal transformations''
refers to this measure.

Mattila \cite[Theorem~6.6]{Mattila1984} proved that
\(\LL^2(rgA-B)>0\) for \(m\times\LL^1\)-almost all
\((g,r)\in\OO\times(0,\infty)\) when
\(\dimH A+\dimH B>2\). For rigid motions, the intersection theorem
in \cite{Mattila1985}, see also \cite[Theorem~13.11]{Mattila}, gives
\(\LL^2(gA-B)>0\) for almost all \(g\) under the conditions
\[
 \dimH A+\dimH B>2,
 \qquad \max\{\dimH A,\dimH B\}>\frac32.
\]

Mattila \cite[Theorem~4.2]{Mattila2022} subsequently obtained this
conclusion under either of the conditions
\begin{equation}\label{eq:previous-area-region}
 \dimH A+\frac12\dimH B>2,
 \qquad \dimH B+\frac12\dimH A>2.
\end{equation}
Thus, both dimensions above \(4/3\) suffice. The same theorem also gives
\[
 \dimH(gA-B)\ge\min\{2,\dimH A+\tfrac12\dimH B\}
\]
and the symmetric bound for almost all \(g\).

The sharp dimension bound was already known when
\(\min\{\dimH A,\dimH B\}\le1/2\), by
\cite[Theorem~4.2(2)]{Mattila2022} and the preceding positive-measure
result. The case \(\dimH A+\dimH B\le1\) also follows from the
projection theorem for general subsets of \(\R^4\) in
\cite[Theorem~3.4]{Mattila2022}.

For the intersections themselves, \cite[Theorem~1.3]{Mattila2024}
shows that if \(\alpha,\beta>0\), \(\alpha+\beta/2>2\),
\(0<\HH^\alpha(A)<\infty\), and \(0<\HH^\beta(B)<\infty\),
then for almost all \(g\) the inequality
\(\dimH(gA\cap(B+z))\ge\alpha+\beta-2\)
holds on a set of translations of positive Lebesgue measure.

It is expected that the sole condition \(\dimH A+\dimH B>2\)
implies \(\LL^2(gA-B)>0\) for almost all \(g\).
The intersection conjecture in \cite[(1.8)]{MattilaSurvey}
would give, for every \(\varepsilon>0\) and almost all \(g\),
intersections of dimension at least
\(\dimH A+\dimH B-2-\varepsilon\) for a set of translations
of positive Lebesgue measure. Our first result gives the sharp
dimension bound for every pair of nonempty Borel sets. In particular,
it gives full dimension when the sum of the dimensions is at least two.
Sharpness is also discussed in Section~\ref{sec:sharpness}. We now state our first main result.

\begin{theorem}\label{thm:full-dimension}
Let \(A,B\subset\R^2\) be nonempty Borel sets. Then for almost all
\(g\in\OO\),
\begin{equation}\label{eq:main-dimension}
 \dimH(gA-B)\ge\min\{\dimH A+\dimH B,2\}.
\end{equation}
Moreover, let \(\mu,\nu\) be compactly supported probability measures
with Frostman exponents \(\alpha,\beta\in(0,2]\), respectively. There
is a Borel set \(\mathcal G\subset\OO\) with \(m(\mathcal G)=1\) such that, for every
\(g\in\mathcal G\) and every Borel set \(S\subset\R^2\),
\[
 \dimH S<\min\{\alpha+\beta,2\}
 \quad\Longrightarrow\quad
 ((g_\#\mu)*\widetilde\nu)(S)=0.
\]
Here \(\widetilde\nu(E)=\nu(-E)\).
\end{theorem}

The proof begins with the Cayley change of coordinates in
Lemma~\ref{fd:lem:cayley}. It associates a line in \(\R^3\) to each
pair \((a,b)\), so that rotated differences become slices of a family
of lines. Mattila \cite[Section~6.7]{Mattila2017} discussed the relation
between such slices, projections and Kakeya type estimates. In our setting,
the product structure of the measure on pairs gives the prism estimate
in Proposition~\ref{fd:prop:product}. This estimate depends separately
on the two transverse widths and excludes the planar concentration
that can occur for general measures on \(\R^4\).

We next sample the weighted line family to apply the shaded-tube
volume estimate of Wang and Zahl \cite[Theorem~1.2]{WZ}, stated as
Theorem~\ref{fd:thm:WZ} below. We also explain below its relation to the
streamlined proof of Guth, Wang and Zahl \cite{GWZ}. This gives a bound for the mass carried by
small sets in rotation--translation space. We apply it to dyadic
covers and obtain a summable bound for the exceptional rotations.
The covers are selected before the set \(S\) is fixed, so one null
set of rotations works for every \(S\) in the measure statement.
Frostman's lemma then gives the conclusion for Borel sets.
The details are in Sections~\ref{sec:preliminaries}
and~\ref{sec:proof-full-dimension}.
For related work on unions of line segments satisfying a
Frostman convex Wolff condition, see Craig \cite{Craig}.

Moreover, for a compactly supported \(t\)-Frostman measure \(\mu\) on
\(\R^2\), Yu \cite[Corollary~1.5]{Yu} proved that
\((g_\#\mu)*\widetilde\mu\) has an \(L^2\) density for almost all
\(g\) when \(t>4/3\). Section~7 of that paper also treats two measures
under the asymmetric condition corresponding to
\eqref{eq:previous-area-region}. Our second theorem lowers the common
threshold for absolute continuity below \(4/3\). Set
\[
 \sbase=\frac{\sqrt{13}-1}{2}=1.302775\ldots.
\]

\begin{theorem}\label{thm:breaking-four-thirds}
Let \(t>\sbase\), and let \(\mu,\nu\) be compactly supported
\(t\)-Frostman probability measures on \(\R^2\). Then for almost all
\(g\in\OO\),
\[
 (g_\#\mu)*\widetilde\nu\ll\LL^2.
\]
In particular, if \(A,B\subset\R^2\) are Borel sets with
\(\min\{\dimH A,\dimH B\}>\sbase\), then
\(\LL^2(gA-B)>0\) for almost all \(g\in\OO\).
\end{theorem}

The proof decomposes each measure according to the directions of
strips through its points. A strip is called heavy when its mass
exceeds a chosen threshold. We first remove short angular intervals
in which the heavy directions concentrate. The remaining heavy
directions satisfy a two-ends condition at every relevant mass
threshold. This leaves three parts: the removed angular intervals,
the light strips, and the remaining heavy strips.

The key step is to estimate the contribution of the two removed parts
in \(L^1\), averaged over rotations. Two short angular intervals
overlap for only a small set of rotations, which makes this term
summable over frequencies. We estimate the other terms in \(L^2\).
Refined decoupling controls the light strips, while a two-ends
Furstenberg estimate bounds the number of heavy strips. Balancing
these estimates gives the condition \(t^2+t-3>0\).
The resulting summability identifies the convolution with an
integrable density for almost all rotations.
Section~\ref{sec:proof-outline} gives the estimates behind this outline.
The same bounds give an averaged $L^{1+\eta_{\rm int}}$ estimate
for the densities, for some $\eta_{\rm int}>0$, see
Corollary~\ref{imp:integrability} for details. The same argument also allows different Frostman exponents for the two measures.
Corollary~6.2 gives a sufficient region which supplements (1.1).

Finally, it is also worth to mention that the first author and Yoo proved related results over prime fields.
By \cite[Theorem~5.2]{PY}, if \(p\equiv3\pmod4\) is prime and
\(A,B\subset\mathbb F_p^2\) have \(|A|,|B|\ge p\), then
\(|gA-B|\gtrsim p^2\) for a positive proportion of
\(g\in\mathrm{SO}(2,\mathbb F_p)\).
By \cite[Theorem~5.1]{PY}, if \(|A|,|B|\ge p^{5/4}\), then
\[
 |gA\cap(B+z)|\sim\frac{|A|\,|B|}{p^2}
\]
for \(\gtrsim p^2\) translations \(z\) and a positive proportion
of \(g\in\mathrm O(2,\mathbb F_p)\). In Section 8, we will prove the continuous analog of this statement under the condition \(\min\{\dimH A,\dimH B\}>\sbase\).

\section{Preliminaries}
\label{sec:preliminaries}

We first fix notation. We then prove the covering and tube estimates
needed for Theorem~\ref{thm:full-dimension}. The main estimate of this
section, Proposition~\ref{fd:prop:testing}, bounds the mass of a
measurable set in the line coordinates introduced below.

We write \(B(x,r)\) for the open ball with centre \(x\) and radius \(r\),
and \(\supp\omega\) for the support of a measure \(\omega\).
If \(T:X\to Y\) is Borel, \(T_\#\omega\) denotes the pushforward of
\(\omega\). For probabilities \(\mu,\nu\) on \(\R^2\), set
\begin{equation}\label{eq:lambda-definition}
 T_g(x,y)=gx-y,\qquad
 \lambda_g=(T_g)_\#(\mu\otimes\nu)
           =(g_\#\mu)*\widetilde\nu,
 \qquad \widetilde\nu(E)=\nu(-E).
\end{equation}
A compactly supported probability \(\omega\) is called
\(t\)-Frostman if
\begin{equation}\label{eq:frostman}
 \omega(B(x,r))\le C_\omega r^{t}
\end{equation}
for \(x\in\R^2,\ 0<r\le1\).
For \(e\in S^1\), let \(e^\perp\) denote its counterclockwise quarter-turn.

The orientation-preserving component of \(\OO\) is \(\SO\), whose elements are
\[
 R_\theta=
 \begin{pmatrix}\cos\theta&-\sin\theta\\
                 \sin\theta& \cos\theta\end{pmatrix}.
\]
We also write $e_\theta=(\cos\theta,\sin\theta)$.
Our Fourier transform is
\[
 \widehat\omega(\xi)=\int e^{-2\pi i x\cdot\xi}\dd\omega(x),
\]
and \(\sigma\) denotes normalized arclength measure on \(S^1\). For
\(r>0\), define
\begin{equation}\label{eq:Er-Sr}
 E_rf(x)=\int_{S^1}e^{2\pi irx\cdot e}f(e)\dd\sigma(e),
 \qquad
 S_\omega(r)=\int_{S^1}|\widehat\omega(re)|^2\dd\sigma(e).
\end{equation}
We fix a smooth function \(\phi_0:[0,\infty)\to[0,1]\) which is equal to
one on \([0,1]\) and vanishes on \([2,\infty)\), and set
\(\phi(s)=\phi_0(s)-\phi_0(2s)\). Then
\begin{equation}\label{eq:dyadic-resolution}
 \phi_0(|\xi|)+\sum_{j=1}^{\infty}\phi(|\xi|/2^j)=1.
\end{equation}
For dyadic \(R\ge2\), \(P_R\) denotes the Fourier multiplier with symbol
\(\phi(|\xi|/R)\).
We write \(X\lesssim Y\) or \(Y\gtrsim X\) if \(X\le CY\), and
\(X\sim Y\) if both inequalities hold. These constants may depend on
fixed measures, Frostman constants, exponents and cutoff functions, but
not on the frequency or spatial scale being summed.

We shall use the following compact form of Frostman's lemma; see
\cite[Chapter~8]{Mattila}. If \(E\subset\R^n\) is Borel and
\(0<t<\dimH E\), then a compact subset of \(E\) supports a
\(t\)-Frostman probability. We shall also need the following covering
lemma. Dyadic squares are half-open, and \(\operatorname{side}(Q)\) denotes their side
length.

\begin{lemma}\label{fd:lem:cover}
Let \(d>0\), let \(K\subset\R^2\) be bounded and let \(S\subset K\) with
\(\dimH S<d\). For every integer \(j_{0}\ge1\), there is a countable
family $\mathcal Q$ of dyadic squares meeting $K$ such that
\[
 S\subset\bigcup_{Q\in\mathcal Q}Q,\qquad
 \operatorname{side}(Q)\le2^{-j_{0}},\qquad
 \sum_{Q\in\mathcal Q}\operatorname{side}(Q)^d\le1.
\]
At each scale \(2^{-j}\), the family therefore contains at most
$\lfloor2^{jd}\rfloor$ squares.
\end{lemma}
\begin{proof}
Since $\HH^d(S)=0$, we may cover $S$ by balls of radii
$r_i\le2^{-j_0-1}$ with
\[
 \sum_i r_i^d\le(16\cdot2^d)^{-1}.
\]
Each ball is covered by at most $16$ dyadic squares whose side length is
between $r_i$ and $2r_i$. We discard squares that do not meet $S$ and
remove repetitions. The resulting family $\mathcal Q$ satisfies
\[
 \operatorname{side}(Q)\le2^{-j_0},\qquad
 \sum_{Q\in\mathcal Q}\operatorname{side}(Q)^d
 \le16\cdot2^d\sum_i r_i^d\le1.
\]
Thus,
\[
 |\{Q\in\mathcal Q:\operatorname{side}(Q)=2^{-j}\}|\,2^{-jd}
 \le\sum_{Q\in\mathcal Q}\operatorname{side}(Q)^d\le1.
\]
This completes the proof.
\end{proof}

We next express rotations in coordinates in which the differences lie on
lines. Set
\[
 \mathsf J=\begin{pmatrix}0&-1\\1&0\end{pmatrix},\qquad
 \mathscr L(a,b)=(p,v):=(a-b,\mathsf J(a+b)).
\]
The map $\mathscr L:\R^4\to\R^4$ is invertible, with
\begin{equation}\label{fd:eq:inverse}
 a=\frac{p-\mathsf Jv}{2},\qquad b=\frac{-p-\mathsf Jv}{2}.
\end{equation}
For the probabilities \(\mu\) and \(\nu\), set
\begin{equation}\label{fd:eq:eta}
 \eta=\mathscr L_\#(\mu\otimes\nu).
\end{equation}

\begin{lemma}\label{fd:lem:cayley}
For $u\in\R$ and $\theta=2\arctan u$,
\begin{equation}\label{fd:eq:cayley}
 (\operatorname{Id}-u\mathsf J)(R_\theta a-b)=p+uv.
\end{equation}
On every bounded interval of $u$, the change of coordinates
\[
 (\theta,z)\longmapsto\bigl(u,(\operatorname{Id}-u\mathsf J)z\bigr)
\]
is smooth and invertible. Both $d\theta/du$ and its spatial Jacobian
are bounded above and below by positive constants.
\end{lemma}
\begin{proof}
The Cayley formula reads
\[
 R_{2\arctan u}
 =\frac{(1-u^2)\operatorname{Id}+2u\mathsf J}{1+u^2}.
\]
Since $\mathsf J^2=-\operatorname{Id}$, multiplication gives
$(\operatorname{Id}-u\mathsf J)R_{2\arctan u}=\operatorname{Id}+u\mathsf J$. Thus,
\[
 (\operatorname{Id}-u\mathsf J)(R_\theta a-b)
 =a-b+u\mathsf J(a+b)=p+uv.
\]
The remaining bounds follow from
\[
 \det(\operatorname{Id}-u\mathsf J)=1+u^2,
 \qquad \frac{d\theta}{du}=\frac{2}{1+u^2}.
\]
This completes the proof.
\end{proof}

For the geometric estimates we may assume that
\begin{equation}\label{fd:eq:normalization}
 \supp\mu,\supp\nu\subset B(0,10^{-3}).
\end{equation}
A common dilation gives this normalization without changing the Frostman
exponents. It only dilates every rotated difference by the same factor.
Set
\[
 \mathcal P=\mathscr L(\supp\mu\times\supp\nu).
\]
Let
\[
 I_0=[-1/4,1/4],\qquad P_u(p,v)=p+uv,\qquad
 \lambda'_u=(P_u)_\#\eta.
\]
Each parameter $(p,v)\in\mathcal {P}$ determines a line in $\R\times\R^2$ and a unit
segment on that line:
\begin{align}
 \ell_{p,v}&=\{(u,p+uv):u\in\R\},\label{fd:eq:line}\\
 \ell^0_{p,v}
 &=\left\{(u,p+uv):|u|\le\frac{1}{2\sqrt{1+|v|^2}}\right\}.
 \label{fd:eq:segment}
\end{align}
The segment is centred at $(0,p)$ and contains
$\{(u,p+uv):|u|\le1/3\}$. Its sufficiently thin neighbourhood is contained in the unit
ball of $\R^3$.
The relation between these line families and restricted projections
is also discussed in \cite[Section~6.7]{Mattila2017}.

Let $\varphi\in C_c^\infty(B(0,1))$ be nonnegative with
$\int_{\R^2}\varphi=1$, and define
\begin{equation}\label{fd:eq:Fdelta}
 F_\delta(u,x)
 =\int\varphi_\delta(x-p-uv)\dd\eta(p,v),
 \qquad \varphi_\delta(x)=\delta^{-2}\varphi(x/\delta).
\end{equation}
Thus, $F_\delta(u,\cdot)=\varphi_\delta*\lambda'_u$ and
\begin{equation}\label{fd:eq:totalmass}
 \int_{I_0}\int_{\R^2}F_\delta(u,x)\dd x\dd u=|I_0|.
\end{equation}
Notice that $F_{\delta}$ is smoothed only in the two spatial variables.
Neighbourhoods of subsets of $\R^3$ are Euclidean.

For a rectangular prism $W\subset\R^3$, define
\[
 \mathcal A(W)=\{(p,v):\ell^0_{p,v}\subset W\}.
\]
There is no restriction on the orientation of $W$.

\begin{lemma}\label{fd:lem:localization}
Suppose \(W\) has dimensions \(r_1\times r_2\times h\), where
\(0<r_1\le r_2\le8\) and \(h\le8\). Then
\(\mathcal A(W)\cap\mathcal P\) has diameter at most \(Cr_2\). In
particular, its preimage under \(\mathscr L\)
is contained in a product of two planar balls of radius $Cr_2$.
\end{lemma}
\begin{proof}
If $r_2$ is bounded below, the diameter bound follows from the boundedness
of $\mathcal {P}$. We may therefore assume that $r_2$ is small and
$\mathcal A(W)\cap\mathcal P$ is nonempty. Let $(p,v)$ belong to this
intersection. The points of its segment with $u=-1/3$ and $u=1/3$ are
at distance at least $2/3$ from each other. Thus,
\[
 h^2+r_1^2+r_2^2\ge4/9.
\]
Since $r_2$ is small, $h$ is bounded below and is the longest side of $W$.
We call the line through the centre of $W$ parallel to that side its
longitudinal axis. The components of the difference of the two segment
points perpendicular to this axis have magnitude at most $Cr_2$. Hence, the axis makes an angle $O(r_2)$ with the segment.
Since $|v|$ is small, the $u$-component of the axis is bounded away from zero, so we can write the axis as $(u,p_0+uv_0)$.

Every point $(u,x)\in W$ is within $Cr_2$ of an axis point
$(u',p_0+u'v_0)$. Since $|v_0|$ is bounded,
\[
 |x-p_0-uv_0|
 \le |x-p_0-u'v_0|+|u'-u|\,|v_0|\le Cr_2.
\]
Applying this at $u=\pm1/3$ on each contained segment gives
\[
 |p-p_0+(v-v_0)/3|+|p-p_0-(v-v_0)/3|\le Cr_2.
\]
Adding and subtracting gives $|p-p_0|+|v-v_0|\le Cr_2$.
The preimage bound follows from \eqref{fd:eq:inverse}. This completes
the proof.
\end{proof}

\begin{proposition}\label{fd:prop:product}
Let $\mu,\nu$ satisfy \eqref{fd:eq:normalization} and the Frostman conditions with exponents $\alpha,\beta\in(0,2]$. For every rectangular prism as in
Lemma~\ref{fd:lem:localization},
\begin{equation}\label{fd:eq:anisotropic}
 \eta(\mathcal A(W))
 \le C\min\{r_1^\alpha r_2^\beta,
             r_1^\beta r_2^\alpha\}.
\end{equation}
Set $s=\min\{\alpha+\beta,2\}$. For $0<\rho\le\min\{r_1,1\}$,
\begin{equation}\label{fd:eq:prism-bound}
 \eta(\mathcal A(W))\le C\rho^{s-2}r_1r_2.
\end{equation}
Moreover,
\begin{equation}\label{fd:eq:parameter-ball}
 \eta(B_{\R^4}(z,r))\le Cr^s
\end{equation}
for $z\in\R^{4},\ 0<r\le1$. All these bounds also hold for any positive
measure $\eta_0\le\eta$.
\end{proposition}
\begin{proof}
We shall evaluate the constraint from the narrowest slab of the prism
at two points of each segment. This gives two independent constraints
on each point of the pair. After fixing one point, these
constraints place the other in a ball of radius $Cr_1$.

By Lemma~\ref{fd:lem:localization}, there are planar balls $B_1,B_2$ of
radius $Cr_2$ such that
\[
 \mathscr L^{-1}(\mathcal A(W)\cap\mathcal P)\subset B_{1}\times B_{2}.
\]

Let $(n_0,n)\in\R\times\R^2$ be a unit normal to the pair of faces
separated by $r_1$. The corresponding slab can be written as
\[
 |n_0u+n\cdot x-c|\le r_1/2.
\]
Substituting $x=p+uv$ at $u=\pm1/3$ yields
\begin{equation}\label{fd:eq:slab-constraints}
 |n\cdot p-c|\le Cr_1,
 \qquad |n_0+n\cdot v|\le Cr_1.
\end{equation}
For sufficiently small $r_1$, the second inequality and the small bound
on $|v|$ imply \(|n|\ge c_0>0\). Otherwise \(|n_0|\) would be close
to one, contradicting that inequality.

For a fixed $a$, inserting $p=a-b$ and $v=\mathsf J(a+b)$ into
\eqref{fd:eq:slab-constraints} gives
\[
 |n\cdot a-n\cdot b-c|\le Cr_1,\qquad
 |n\cdot\mathsf Jb+(n_0+n\cdot\mathsf Ja)|\le Cr_1.
\]
The coefficient matrix is
\[
 M_n=\begin{pmatrix}-n^{\mathsf T}\\ n^{\mathsf T}\mathsf J\end{pmatrix}.
\]
Its two rows are perpendicular and have norm $|n|$, so
$|\det M_n|=|n|^2\ge c_0^2$ and its inverse has uniformly bounded norm.
For this fixed $a$, all admissible $b$ therefore lie in a ball of radius
$Cr_1$. Integrating the $\beta$-Frostman estimate over the ball containing
the \(a\)-coordinates, of radius $Cr_2$, gives
\[
 \eta(\mathcal A(W))\le Cr_1^\beta r_2^\alpha.
\]
Fixing $b$ instead and using the two independent constraints on $a$
gives $Cr_1^\alpha r_2^\beta$. When $r_1$ is bounded below, both estimates
follow from total mass one after changing the constant. This proves
\eqref{fd:eq:anisotropic}.

To prove \eqref{fd:eq:prism-bound}, we may assume that $\alpha\ge\beta$,
since \eqref{fd:eq:anisotropic} is symmetric in the exponents.
Suppose first that $\alpha+\beta\le2$, so $s=\alpha+\beta$.
Let $x_1=r_1/\rho$ and $x_2=r_2/\rho$. Then $1\le x_1\le x_2$ and
\[
 \rho^{-s}r_1^\alpha r_2^\beta=x_1^\alpha x_2^\beta.
\]
If $\alpha\le1$, then $\beta\le1$ and this is at most $x_1x_2$.
If $\alpha>1$, then
\[
 \frac{x_1^\alpha x_2^\beta}{x_1x_2}
 =x_1^{\alpha-1}x_2^{\beta-1}
 \le x_2^{\alpha+\beta-2}\le1.
\]
Thus, \eqref{fd:eq:anisotropic} gives \eqref{fd:eq:prism-bound}.

If $\alpha+\beta\ge2$, then $s=2$ and $\alpha\ge1$. For
$0<r_1\le r_2\le1$,
\[
 \frac{r_1^\alpha r_2^\beta}{r_1r_2}
 =\left(\frac{r_1}{r_2}\right)^{\alpha-1}
   r_2^{\alpha+\beta-2}\le1.
\]
For $r_2\le8$, the same expression is bounded by a fixed constant.
This proves \eqref{fd:eq:prism-bound} also in this case.
The preimage of a parameter ball under $\mathscr L$ is contained in a
product of two balls of radius $Cr$. Hence, its product mass is at most
$Cr^{\alpha+\beta}\le Cr^s$ for $r\le1$, proving
\eqref{fd:eq:parameter-ball}. Every positive measure $\eta_0\le\eta$
inherits the same upper bounds. This completes the proof.
\end{proof}

A $\rho$-tube is the closed $\rho$-neighbourhood of a unit line segment
in $\R^3$; its volume is comparable to $\rho^2$. A family of such tubes
is \emph{essentially distinct} if
\[
 |T\cap T'|\le\tfrac12\max\{|T|,|T'|\}
\]
for distinct tubes $T,T'$. A \emph{shading} of a tube \(T\) is a
measurable set \(Y(T)\subset T\).

We use the following theorem of Wang and Zahl, in the preprint version
\cite[Theorem~1.2]{WZ}.

\begin{theorem}[Wang--Zahl]\label{fd:thm:WZ}
For every $\varepsilon>0$, there exist $K_{\rm WZ}>1$ and $\rho_\varepsilon>0$
with the following property. Let $0<\rho<\rho_\varepsilon$, and let
$\cT$ be a finite, essentially distinct family of $\rho$-tubes contained
in the unit ball of $\R^3$. Suppose that every rectangular prism of
dimensions $r_1\times r_2\times2$ contains at most
$100r_1r_2\rho^{-2}$ tubes of $\cT$. If $0<\lambda\le1$ and
$|Y(T)|\ge\lambda|T|$ for every $T\in\cT$, then
\begin{equation}\label{fd:eq:WZ}
 \left|\bigcup_{T\in\cT}Y(T)\right|
 \ge\rho^\varepsilon\lambda^{K_{\rm WZ}}\sum_{T\in\cT}|T|.
\end{equation}
\end{theorem}

The same estimate also follows from Guth, Wang and Zahl
\cite[Theorem~1.1]{GWZ}. By John's ellipsoid theorem and the unit-ball
support, the prism hypothesis gives
\[
 \sum_{T\in\cT:T\subset K}|T|\lesssim |K|
 \quad\text{for every convex set }K,
 \qquad |\cT|\lesssim\rho^{-2}.
\]
Apply their theorem with loss $\varepsilon/2$, and let
$\eta_{\rm GWZ}>0$ be its density parameter. If
$\lambda\ge\rho^{\eta_{\rm GWZ}}$, their conclusion gives
\eqref{fd:eq:WZ} for sufficiently small $\rho$. Otherwise, choose
$K_{\rm WZ}>1+2/\eta_{\rm GWZ}$. For any $T_0\in\cT$,
\[
 \rho^\varepsilon\lambda^{K_{\rm WZ}}\sum_{T\in\cT}|T|
 \le C\rho^{\varepsilon-2+\eta_{\rm GWZ}(K_{\rm WZ}-1)}
 \lambda|T_0|
 \le |Y(T_0)|
\]
when $\rho$ is sufficiently small. This proves the estimate in the
remaining case.

We shall apply this theorem with a fixed $\varepsilon>0$, so no uniform
bound for $K_{\rm WZ}(\varepsilon)$ as $\varepsilon\downarrow0$ is needed.
The next lemma allows us to apply it to a weighted family of tubes.

Fix $s\in(0,2]$. Let $\eta_0$ be a positive measure supported on
$\mathcal P$, with mass at most one. For $0<\delta<1$, we partition
$\R^{4}$ into half-open cubes $Q_i$ of side $\delta$. In each cube of positive mass we choose
$q_i=(p_i,v_i)\in Q_i\cap\supp\eta_0$ and set $w_i=\eta_0(Q_i)$.
Let $T_i$ be the $\rho$-tube about $\ell^0_{p_i,v_i}$, where
$\rho=C_{\rm fd}\delta$. We choose $C_{\rm fd}$ large enough that
\[
 \{(u,x):u\in I_0,\ |x-p-uv|\le\delta\}\subset T_{i}
 \quad\text{for }(p,v)\in Q_{i}\cap\mathcal P.
\]
This is possible because $|p+uv-p_i-uv_{i}|\le C\delta$ for $u\in I_0$,
and these values of $u$ stay away from the segment ends.
For a set of indices $G$, let $a_G=\sum\limits_{i\in G}w_i$.

\begin{lemma}\label{fd:lem:selection}
Suppose that \eqref{fd:eq:prism-bound} and \eqref{fd:eq:parameter-ball}
hold with $\eta_0$ in place of $\eta$, for the fixed exponent $s$
and a fixed constant. If $\delta$ is sufficiently small and
$a_G\ge\delta^{s/2}$, then there is
$\mathcal S\subset G$ such that
\begin{equation}\label{fd:eq:selection-size}
 |\mathcal S|\ge
 c\,\frac{a_G\rho^{-s}}{(\log(1/\rho))^4},
\end{equation}
and the tubes $\{T_i:i\in\mathcal S\}$ satisfy all the geometric
hypotheses of Theorem~\ref{fd:thm:WZ}. The constants and the upper bound
on $\delta$ may depend on $s$ and the constant in the two assumed bounds,
but are independent of $G$ and of the measure $\eta_0$ within those bounds.
\end{lemma}
\begin{proof}
We first sample the tubes independently according to their weights.
The prism estimate controls the number selected in each sufficiently
large prism. We then remove tubes with nearby parameters. This gives
the required bound for smaller prisms and makes the remaining tubes
essentially distinct.

Let $L_{\rho}=\log(1/\rho)$. We select the indices in $G$ independently with
probabilities
\begin{equation}\label{fd:eq:probability}
 \mathbb P(i\text{ is selected})
 =\frac{w_i}{C_2\rho^sL_{\rho}^4}.
\end{equation}
The ball bound gives $w_i\le C\delta^s$, so these are valid probabilities
for $C_2$ fixed sufficiently large and $\delta$ sufficiently small.
If $X$ is the number selected, then
\[
 M_0:=\mathbb EX=\frac{a_G}{C_2\rho^sL_{\rho}^4},\qquad
 \operatorname{Var}X\le M_0.
\]
Since $a_G\ge\delta^{s/2}$ and $\rho=C_{\rm fd}\delta$,
\[
 M_0\ge c_s\delta^{-s/2}(\log(1/\delta))^{-4}\longrightarrow\infty.
\]
Here $s>0$ is fixed, and the convergence is uniform in $G$ and
$\eta_0$ under the stated bounds. By Chebyshev's inequality,
\begin{equation}\label{fd:eq:sample-lower}
 \mathbb P(X<M_0/2)\le4/M_0=o(1).
\end{equation}

We first choose a collection of at most $\rho^{-N_0}$ test prisms, for a fixed
absolute $N_0$, with the following property. Every
$r_1\times r_2\times2$ prism $W$ containing a $\rho$-tube in the unit
ball has an associated test prism $W'$ such that
\begin{equation}\label{fd:eq:net}
 W\cap B(0,1)\subset W',\qquad
 r'_1\le4r_1,\quad r'_2\le4r_2,
\end{equation}
where the transverse widths are ordered so that
$\rho\le r'_1\le r'_2\le4$, and the length of $W'$ is at most six.
To construct this collection, we first truncate widths exceeding two when
covering $W\cap B(0,1)$, and we replace the longitudinal constraint by an
interval of length two containing the projection of the unit ball. The
resulting centres lie in a bounded set. We then round the transverse widths
upward to dyadic values and approximate the centres and orthonormal frames
by finite nets of accuracy $\rho^3$. Finally, we double the rounded widths
and enlarge the length to six. Because the approximation error is
$O(\rho^3)$ and the relevant widths are at least $2\rho$ whenever $W$
contains a $\rho$-tube, this construction gives \eqref{fd:eq:net}. The
centre and frame nets have polynomial cardinality in $\rho^{-1}$, while the
two dyadic width choices contribute only a logarithmic factor. Thus, after
increasing $N_0$, the total cardinality is at most $\rho^{-N_0}$.

For each test prism $W'$, denote by $X_{W'}$ the number of selected tubes
contained in it. The sum of the weights $w_i$ over all indices
$i\in G$ with $T_i\subset W'$ is at most
$C\rho^{s-2}r'_1r'_2$. To see this, notice that the segments depend Lipschitz
continuously on $(p,v)$ in the bounded parameter region. Hence, all
parameters in the cubes of those representatives have their segments in
an $O(\delta)$ enlargement of $W'$. For sufficiently small $\delta$, this enlargement has transverse
widths comparable to $r'_1,r'_2$, and all its dimensions are at most eight.
The prism estimate then applies. With
$k_{W'}=r'_1r'_2\rho^{-2}$, we obtain
\begin{equation}\label{fd:eq:prism-expectation}
 \mathbb EX_{W'}\le\frac{Ck_{W'}}{C_2L_{\rho}^4}.
\end{equation}
For a sum $X_{\rm B}$ of independent Bernoulli variables with mean
$M_{\rm B}>0$, we have
\begin{equation}\label{fd:eq:bernoulli}
 \mathbb P(X_{\rm B}\ge k)\le(eM_{\rm B}/k)^k
\end{equation}
for $k\ge M_{\rm B}$. Independence gives
\[
 \mathbb E e^{s_{\rm B}X_{\rm B}}
 \le\exp\bigl((e^{s_{\rm B}}-1)M_{\rm B}\bigr)
 \quad\text{for }s_{\rm B}\ge0.
\]
The exponential Markov inequality with $e^{s_{\rm B}}=k/M_{\rm B}$ gives
\[
 \mathbb P(X_{\rm B}\ge k)
 \le\exp\bigl(-s_{\rm B}k+(e^{s_{\rm B}}-1)M_{\rm B}\bigr)
 \le(eM_{\rm B}/k)^{k}.
\]
The claim is immediate when $M_{\rm B}=0$.

We choose $C_{5}>N_0+10$. If $k_{W'}\ge C_{5}L_{\rho}$, then
\eqref{fd:eq:bernoulli}, applied with $k=\lceil k_{W'}\rceil$, together
with \eqref{fd:eq:prism-expectation}, shows that the failure probability
is at most $e^{-k_{W'}}$ for sufficiently small $\rho$. A union bound
over the test prisms has total failure probability at most
\[
 \rho^{-N_{0}}e^{-C_{5}L_{\rho}}=\rho^{C_{5}-N_{0}}=o(1).
\]
Thus, with probability $1-o(1)$,
\begin{equation}\label{fd:eq:large-counts}
 X_{W'}\le2k_{W'}
\end{equation}
for $k_{W'}\ge C_{5}L_{\rho}$.

This leaves the prisms with $k_{W'}<C_{5}L_{\rho}$. We treat them by
separating nearby parameters. Let
\[
 r_0=C_3\rho L_{\rho},
\]
where $C_3$ is large relative to $C_{5}$ and the constant in
Lemma~\ref{fd:lem:localization}. Let $N_{\rm pair}$ be the number of
unordered pairs of distinct selected indices with $|q_i-q_j|\le r_0$.
If $|q_i-q_j|\le r_0$, then $Q_j\cap\mathcal P\subset B(q_i,r_0+C\delta)$.
Thus, the parameter ball estimate gives
\begin{align*}
 \mathbb E N_{\rm pair}
 &\le\frac{C}{C_2^2\rho^{2s}L_{\rho}^8}
       \sum_{i\in G}w_i
       \sum_{j:\,|q_i-q_j|\le r_0}w_j\\
 &\le\frac{C a_G r_0^s}{C_2^2\rho^{2s}L_{\rho}^8}
 \le C_s M_0L_{\rho}^{s-4}=o(M_0).
\end{align*}
Since $s\le2$, the last logarithmic power tends to zero.
By Markov's inequality, \(N_{\rm pair}\le M_0/8\) with probability \(1-o(1)\).
The preceding probability estimates allow us to choose an outcome for
which \(N_{\rm pair}\le M_0/8\), \(X\ge M_0/2\), and
\eqref{fd:eq:large-counts} all hold. From this outcome, we remove every selected index belonging to a pair at distance at most $r_0$. At most $2N_{\rm pair}\le M_0/4$ indices are removed, so at least $M_0/4$ remain. Their parameters are $r_0$-separated. If a test prism has $k_{W'}<C_{5}L_{\rho}$, then $r'_1\ge\rho$ implies $r'_2<C_{5}\rho L_{\rho}$.
Lemma~\ref{fd:lem:localization} shows that the parameters of its contained segments have diameter at most $CC_{5}\rho L_{\rho}<r_0$. Such a prism contains
at most one remaining tube.

For an arbitrary $r_1\times r_2\times2$ prism, we compare it with its
associated test prism. In the large case, the number of remaining tubes is
at most
\[
 2r'_1r'_2\rho^{-2}\le32r_1r_2\rho^{-2}
 \le100r_1r_2\rho^{-2}.
\]
In the small case it is at most one, and the same final bound holds:
a prism containing a $\rho$-tube has both transverse widths at least
$2\rho$.

It remains to prove essential distinctness. For parameters in \(\mathcal P\),
the intersection of any tube with
a time slice has area at most $C\rho^2$. If two tubes intersected in
more than half their common volume, the set of times at which their
slices meet would have length bounded below by an absolute constant.
At each such time,
\[
 |p_i-p_j+u(v_i-v_j)|\le C\rho.
\]
Here the same estimate also holds near the rounded tube ends, since a
point in a tube is within $\rho$ of an axis point, and the slopes are
bounded. We may choose two such times $u_1,u_2$ with $|u_{1}-u_{2}|\ge c>0$. Then
\[
 |u_{1}-u_{2}|\,|v_{i}-v_{j}|\le C\rho,
 \qquad |p_i-p_j|\le C\rho+|u_{1}|\,|v_i-v_j|\le C\rho.
\]
This contradicts $r_0$-separation for small
$\rho$. Thus, the remaining tubes are essentially distinct.
Their containment in the unit ball follows from \eqref{fd:eq:normalization}
and \eqref{fd:eq:segment}. Their number is at least $M_0/4$, so
\eqref{fd:eq:selection-size} follows.
This completes the proof.
\end{proof}

We now use the selected tubes to estimate the mass of a measurable set.

\begin{proposition}\label{fd:prop:testing}
Let $\mu,\nu$ be compactly supported Frostman probabilities with exponents
$\alpha,\beta\in(0,2]$, and suppose that \eqref{fd:eq:normalization} holds.
Set $s=\min\{\alpha+\beta,2\}$. For every $\varepsilon>0$ there are
$Q_{\rm test}=Q_{\rm test}(\varepsilon)>2$, $C_\varepsilon<\infty$, and
$\delta_\varepsilon>0$ such that, for every measurable
$\Omega\subset I_0\times\R^2$ and $0<\delta<\delta_\varepsilon$,
\begin{equation}\label{fd:eq:testing}
 \int_\Omega F_\delta(u,x)\dd u\dd x
 \le C_\varepsilon\left(
        \delta^{s/2}+\bigl(\delta^{-(2-s)-\varepsilon}|\Omega|\bigr)^{1/Q_{\rm test}}\right).
\end{equation}
The exponent $Q_{\rm test}$ depends only on $\varepsilon$. The constants
and $\delta_\varepsilon$ may also depend on the fixed Frostman exponents
and constants, but are independent of $\Omega$ and $\delta$.
\end{proposition}
\begin{proof}
If the left-hand side of \eqref{fd:eq:testing} is small, the first term
on the right suffices. Otherwise, a positive amount of the total weight
lies on tubes for which $\Omega$ occupies a definite proportion of the
volume. We apply Lemma~\ref{fd:lem:selection} to those tubes and use the
Wang--Zahl theorem to bound $|\Omega|$ from below.

We use the cubes $Q_i$, weights $w_i$ and tubes $T_i$ defined before
Lemma~\ref{fd:lem:selection}, taking $\eta_0=\eta$ from
\eqref{fd:eq:eta}. Proposition~\ref{fd:prop:product} verifies its hypotheses.
The support and upper bound of $\varphi$ give, on $I_0\times\R^2$,
\begin{equation}\label{fd:eq:domination}
 F_\delta(u,x)\le C\delta^{-2}\sum_i w_i\one_{T_i}(u,x).
\end{equation}
For each $(p,v)\in Q_{i}$ with $\varphi_\delta(x-p-uv)\ne0$, the choice
of $C_{\rm fd}$ gives $(u,x)\in T_{i}$.

We set
\[
 a_{\Omega}=\int_\Omega F_\delta,\qquad
 d_i=\frac{|\Omega\cap T_i|}{|T_i|}\in[0,1].
\]
Since $|T_i|\asymp\delta^2$, \eqref{fd:eq:domination} gives
\begin{equation}\label{fd:eq:mass-richness}
 a_{\Omega}\le C_4\sum_i w_i d_i.
\end{equation}
Here we take $C_4\ge1$. If $a_{\Omega}\le2C_4\delta^{s/2}$, the first term in \eqref{fd:eq:testing} suffices.
Otherwise, let
\[
 G=\{i:d_i\ge a_{\Omega}/(2C_4)\}.
\]
Because $\sum\limits_iw_i=1$, the complementary indices contribute at most
$a_{\Omega}/(2C_4)$ to $\sum\limits_iw_id_i$. It follows that
\begin{equation}\label{fd:eq:good-mass}
 a_G=\sum_{i\in G}w_i
 \ge\sum_{i\in G}w_i d_i\ge a_{\Omega}/(2C_4)>\delta^{s/2}.
\end{equation}
By Lemma~\ref{fd:lem:selection}, there is $\mathcal S\subset G$ with
\eqref{fd:eq:selection-size} whose tubes satisfy the hypotheses of
Theorem~\ref{fd:thm:WZ}. We apply that theorem with $\varepsilon/2$ in
place of $\varepsilon$ and with $Y(T_i)=\Omega\cap T_i$. These shadings
have relative volume at least $\lambda=a_{\Omega}/(2C_4)$.
By \eqref{fd:eq:totalmass}, $0<a_{\Omega}\le|I_0|=1/2$, so
$0<\lambda\le1$. If $K_{\rm WZ}>1$ is the exponent in that theorem, we obtain
\begin{align*}
 |\Omega|
 &\ge\left|\bigcup_{i\in\mathcal S}Y(T_i)\right|\\
 &\ge\rho^{\varepsilon/2}
       \left(\frac{a_{\Omega}}{2C_4}\right)^{K_{\rm WZ}}
       \sum_{i\in\mathcal S}|T_i|\\
 &\ge c_\varepsilon\delta^{2-s+\varepsilon/2}
       a_{\Omega}^{K_{\rm WZ}}\frac{a_G}{(\log(1/\rho))^4}
 \ge c'_\varepsilon\delta^{2-s+\varepsilon}a_{\Omega}^{K_{\rm WZ}+1}.
\end{align*}
Here we used $|T_i|\asymp\rho^2$ and \eqref{fd:eq:selection-size}.
The last step uses
\eqref{fd:eq:good-mass} and the inequality
\[
 (\log(1/\rho))^{-4}\ge c_\varepsilon\delta^{\varepsilon/2},
\]
which holds for sufficiently small $\delta$.
We set $Q_{\rm test}=K_{\rm WZ}+1>2$ and solve the preceding inequality for
$a_{\Omega}$. Combining the resulting bound with the case
$a_{\Omega}\le2C_4\delta^{s/2}$ proves \eqref{fd:eq:testing}. This completes the proof.
\end{proof}

\begin{remark}\label{fd:rem:dimension-not-area}
The Wang--Zahl theorem is the only non-elementary input in the proof of
Theorem~\ref{thm:full-dimension}. For fixed $\varepsilon$, the estimate
\eqref{fd:eq:testing} has a fixed positive exponent $1/Q_{\rm test}$.
A cover with at most $\delta^{-d}$ squares of side $\delta$, where $d<s$,
gives $|\Omega|\lesssim\delta^{2-d}$,
so the second term is at most $C\delta^{(s-d-\varepsilon)/Q_{\rm test}}$.
We may choose $\varepsilon<s-d$ for each fixed $d$. No uniform bound for
$Q_{\rm test}$ as $d\uparrow s$ is needed. At $s=2$, the estimate still
does not give positive Lebesgue measure.
\end{remark}
\section{Proof of Theorem \ref{thm:full-dimension}}
\label{sec:proof-full-dimension}

We shall first prove a measure version of Theorem~\ref{thm:full-dimension}.
The testing estimate from the preceding section gives one exceptional
set of rotations for all Borel sets of dimension less than
$\min\{\alpha+\beta,2\}$. Frostman's lemma then gives the set statement.

\begin{proposition}
\label{prop:full-dimension-measure}
Let \(\mu,\nu\) be compactly supported Frostman probability measures on
\(\R^2\) with exponents \(\alpha,\beta\in(0,2]\), respectively. Set
\(s=\min\{\alpha+\beta,2\}\). There is a Borel set \(\mathcal G\subset\OO\),
with \(m(\OO\setminus\mathcal G)=0\), such that for every
\(g\in\mathcal G\) and every Borel set \(S\subset\R^2\) with
\(\dimH S<s\),
\[
 \lambda_g(S)=0.
\]
\end{proposition}

\begin{proof}[Proof of Proposition~\ref{prop:full-dimension-measure}]
For each $d<s$, we control simultaneously all unions of at most
$2^{jd}$ dyadic squares at scale $2^{-j}$. The testing estimate and
Borel--Cantelli give a null set outside which the mass of every such
union tends to zero at a summable rate. The covering lemma then gives
the conclusion for every Borel set of dimension less than $d$.

We first prove the claim in the line coordinates, assuming
\eqref{fd:eq:normalization}. All the probabilities $\lambda'_u$,
$u\in I_0$, are supported in a fixed closed bounded square $K_0\subset\R^2$.
We fix a rational number $d\in(0,s)$ and set $\delta_j=2^{-j}$. We denote
by $\cD_j$ the finite family of dyadic squares of side $\delta_j$ meeting
$K_0$. We define
\begin{equation}\label{fd:eq:Hj}
 H_j(u)=\max\left\{
 \lambda'_u\left(\bigcup\mathcal Q\right):
 \mathcal Q\subset\cD_j,\quad
 |\mathcal Q|\le\lfloor2^{jd}\rfloor\right\},
\end{equation}
and
\[
 E_j=\{u\in I_0:H_j(u)>j^{-2}\}.
\]
For each fixed Borel square $Q$, the map
\[
 u\longmapsto\lambda'_u(Q)
 =\int\one_Q(p+uv)\dd\eta(p,v)
\]
is Borel. The same holds for a finite union of squares. Since the
maximum in \eqref{fd:eq:Hj} is over finitely many families, $H_j$ and $E_j$
are Borel.

We enumerate those finite families and, for each $u\in E_j$, denote by
$\mathcal Q_j(u)$ the first family attaining the maximum in
\eqref{fd:eq:Hj}. The parameter set on which any particular family is
chosen is Borel. We define $\Omega_j\subset I_0\times\R^2$ by its sections:
\[
 (\Omega_j)_u=
 \begin{cases}
 N_{\delta_j}\!\left(\displaystyle\bigcup\mathcal Q_j(u)\right),
       &u\in E_j,\\[2mm]
 \varnothing,&u\notin E_j.
 \end{cases}
\]
Here $N_r$ denotes the open $r$-neighbourhood in $\R^2$.
$\Omega_j$ is Borel because it is a finite union of products of Borel
parameter sets with fixed open subsets of $\R^2$. Each section is a union
of at most $2^{jd}$ enlarged squares of area at most $C\delta_j^2$. Therefore,
\begin{equation}\label{fd:eq:Omega-size}
 |\Omega_j|\le C\delta_j^{2-d}.
\end{equation}

For \(u\in E_j\), the nonnegativity of \(\varphi_{\delta_j}\), its support
in \(B(0,\delta_j)\), and the identity
\(\int\varphi_{\delta_j}=1\) give
\[
 \int_{(\Omega_j)_u}F_{\delta_j}(u,x)\dd x
 \ge \lambda'_u\left(\bigcup\mathcal Q_j(u)\right)=H_j(u).
\]
Hence,
\begin{equation}\label{fd:eq:Omega-mass}
 \int_{\Omega_j}F_{\delta_j}\dd u\dd x
 \ge\int_{E_j}H_j(u)\dd u
 \ge j^{-2}\LL^1(E_j).
\end{equation}

We apply Proposition~\ref{fd:prop:testing} with
$\varepsilon=(s-d)/2$ and fix the resulting exponent
$Q_{\rm test}=Q_{\rm test}((s-d)/2)$. Since
\[
 \bigl(\delta_j^{-(2-s)-\varepsilon}\delta_j^{2-d}\bigr)^{1/Q_{\rm test}}
 =\delta_j^{(s-d)/(2Q_{\rm test})},
\]
\eqref{fd:eq:Omega-size} and \eqref{fd:eq:Omega-mass} give, for all
sufficiently large $j$,
\begin{equation}\label{fd:eq:exceptional-sum}
 \LL^1(E_j)
 \le C_{s,d}j^2\left(2^{-js/2}+2^{-j(s-d)/(2Q_{\rm test})}\right).
\end{equation}
For the fixed $s$ and $d$, both series on the right are summable.
Their constants need not be uniform as $d\uparrow s$. Borel--Cantelli
now gives a Borel null set $Z_d\subset I_0$ such that, for every
$u\in I_0\setminus Z_d$ and all sufficiently large $j$,
\begin{equation}\label{fd:eq:small-covers}
 H_j(u)\le j^{-2}.
\end{equation}

We now fix $u\notin Z_d$ and a Borel set $S\subset\R^2$ with $\dimH S<d$.
Since $\lambda'_u$ is supported in $K_0$, it is enough to consider $S\cap K_0$.
For each sufficiently large $j_0$, Lemma~\ref{fd:lem:cover} provides a dyadic
cover of $S\cap K_0$ at scales $2^{-j}$, $j\ge j_{0}$, with at most
$\lfloor2^{jd}\rfloor$ squares at each scale. By \eqref{fd:eq:Hj} and
\eqref{fd:eq:small-covers},
\[
 \lambda'_u(S)=\lambda'_u(S\cap K_0)
 \le\sum_{j\ge j_{0}}H_j(u)
 \le\sum_{j\ge j_{0}}j^{-2}.
\]
Letting $j_{0}\to\infty$ proves $\lambda'_u(S)=0$.
The set $Z_d$ was fixed before choosing $S$, so this conclusion holds
simultaneously for every such Borel set $S$.

We set
\[
 Z=\bigcup_{d\in\mathbb Q\cap(0,s)}Z_d.
\]
It is a Borel null set. For every $u\in I_0\setminus Z$ and every Borel
$S$ with $\dimH S<s$, we choose \(d\in\mathbb Q\cap(0,s)\) with
\(\dimH S<d<s\).
The preceding argument then gives
\begin{equation}\label{fd:eq:slice-result}
 \lambda'_u(S)=0.
\end{equation}
Thus, the claim holds for $\lambda'_u$ outside one Borel null set.

We now return to the rotations. Set $\theta=2\arctan u$. By
Lemma~\ref{fd:lem:cayley},
\begin{equation}\label{fd:eq:pushforward-chart}
 \lambda'_u=(\operatorname{Id}-u\mathsf J)_\#\lambda_{R_\theta}.
\end{equation}
The map $\operatorname{Id}-u\mathsf J$ is an invertible linear map of the plane.
It carries Borel sets to Borel sets and preserves Hausdorff dimension.
For every Borel set $S$,
\[
 \lambda_{R_\theta}(S)
 =\lambda'_u\bigl((\operatorname{Id}-u\mathsf J)S\bigr).
\]
Thus, \eqref{fd:eq:slice-result} implies that $\lambda_{R_\theta}$ gives
zero mass to every Borel set of dimension less than $s$ for almost all
\(\theta\in[-2\arctan(1/4),2\arctan(1/4)]\).
The exceptional angles form a null set because $d\theta/du$ is bounded
above and below on $I_0$.

Finitely many translated angle intervals cover all rotations. On an
interval centred at $\theta_0$, we have
\[
 R_\theta=R_{\theta-\theta_0}R_{\theta_0}.
\]
We apply the same argument to $(R_{\theta_0})_\#\mu$ and $\nu$.
Rotations preserve the support normalization and the Frostman bounds.
The finite union of the exceptional angle sets is null, so the claim
holds for almost all elements of $\SO$.

We fix the reflection
$g_0=\operatorname{diag}(1,-1)$. Every orientation-reversing
$g\in\OO$ has the form $R_{\theta} g_0$. We apply the rotation result
to $(g_0)_\#\mu$ and $\nu$. Since Haar measure on $\OO$ gives each
component half its mass and is normalized angular measure on either
component, this proves Proposition~\ref{prop:full-dimension-measure} for
normalized supports.
A common dilation, as described before \eqref{fd:eq:normalization}, removes
that normalization. All exceptional sets constructed above are Borel,
so $\mathcal G$ may be chosen Borel. This completes the proof.
\end{proof}

\begin{proof}[Proof of Theorem~\ref{thm:full-dimension}]
The measure statement is Proposition~\ref{prop:full-dimension-measure}.
For the set statement, we apply this proposition to Frostman measures
whose exponents approach the dimensions of $A$ and $B$.
We first treat the cases where one of these dimensions is zero.
If $\dimH A=0$, then for any $a\in A$ and every $g\in\OO$,
\[
 ga-B\subset gA-B,\qquad
 \dimH(ga-B)=\dimH B=\min\{\dimH A+\dimH B,2\}.
\]
If $\dimH B=0$, we use $gA-b\subset gA-B$ for any $b\in B$.
Thus, the claim holds for every $g$ in these cases.

Suppose that both dimensions are positive. We choose positive sequences
$\alpha_n\uparrow\dimH A$ and $\beta_n\uparrow\dimH B$, with
$\alpha_n<\dimH A$ and $\beta_n<\dimH B$.
By the compact form of Frostman's lemma, there are compact sets
$A_n\subset A$, $B_n\subset B$ supporting Frostman probabilities
$\mu_n,\nu_n$ with exponents $\alpha_n,\beta_n$, respectively.
Proposition~\ref{prop:full-dimension-measure} gives a Borel set
$\mathcal G_n\subset\OO$ of full measure for each pair.

Let $g\in\bigcap_{n=1}^{\infty}\mathcal G_n$. The compact set $gA_n-B_n$
has mass one for $(g_\#\mu_n)*\widetilde\nu_n$, so the proposition gives
\[
 \dimH(gA-B)\ge\dimH(gA_n-B_n)
 \ge\min\{\alpha_n+\beta_n,2\}
 \quad\text{for every }n.
\]
Letting $n\to\infty$, we obtain
\[
 \dimH(gA-B)\ge\min\{\dimH A+\dimH B,2\}.
\]
The intersection of the sets $\mathcal G_n$ has full measure. In particular,
the endpoint $\dimH A+\dimH B=2$ is included, without a hypothesis on
the critical Hausdorff measures. This completes the proof.
\end{proof}

The slab constraints in the proof of
Proposition~\ref{fd:prop:product} have a simple geometric interpretation.
If the lines lie in a fixed plane $n_0u+n\cdot x=c$, then $n\ne0$ and
\[
 n\cdot(a-b)=c,\qquad n\cdot\mathsf J(a+b)=-n_0.
\]
These equations are equivalent to $b=\Psi(a)$, where $\Psi$ is an
orientation-reversing isometry. The following example corresponds to
a reflection.

The product hypothesis in Proposition~\ref{prop:full-dimension-measure}
is essential. Let $g_0=\operatorname{diag}(1,-1)$ and
\[
 E=\{(a,g_0a):a\in\overline{B(0,10^{-3})}\}\subset\R^4.
\]
This set has Hausdorff dimension two and supports a $2$-Frostman
probability measure. For every $\theta$,
\[
 \det(R_\theta-g_0)
 = (\cos\theta-1)(\cos\theta+1)+\sin^2\theta=0.
\]
Since a rotation cannot equal a reflection, $R_\theta-g_0$ has rank one.
Thus, $T_{R_\theta}(E)$ lies in a line for every rotation, and
$\SO$ has Haar measure $1/2$ in $\OO$. In Cayley coordinates, if
$a=(a_1,a_2)$ and $b=g_0a$, then
\[
 p=(0,2a_2),\qquad v=(0,2a_1),\qquad
 \ell_{p,v}=\{(u,0,2a_2+2ua_1):u\in\R\}.
\]
All these lines lie in the plane $x_1=0$.
The prism estimate \eqref{fd:eq:prism-bound} excludes this
concentration for product measures whenever $s>1$. Indeed, the segments
in this plane lie in prisms $W$ with transverse widths $\rho$ and one,
while \eqref{fd:eq:prism-bound} gives
$\eta(\mathcal A(W))\le C\rho^{s-1}\to0$ for product measures.

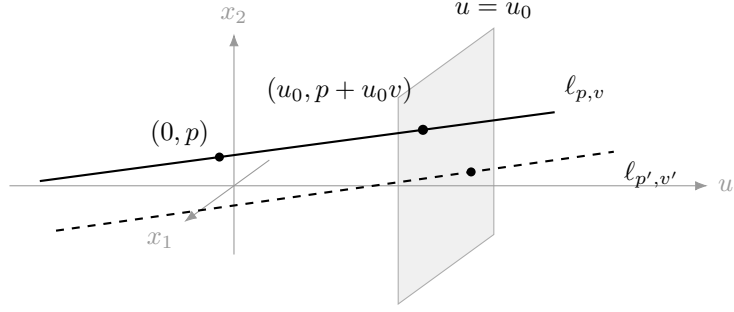
\begin{figure}[htbp]
\centering
\begin{tikzpicture}[x={(1.65cm,0cm)},y={(-0.55cm,-0.40cm)},
                    z={(0cm,1.30cm)},>=Latex,font=\small]
  \fill[gray!12] (1.7,-1.15,-0.85)--(1.7,1.15,-0.85)
       --(1.7,1.15,1.25)--(1.7,-1.15,1.25)--cycle;
  \draw[gray!65] (1.7,-1.15,-0.85)--(1.7,1.15,-0.85)
       --(1.7,1.15,1.25)--(1.7,-1.15,1.25)--cycle;
  \node[above] at (1.7,-1.15,1.25) {$u=u_0$};
  \draw[->,gray!80] (-1.8,0,0)--(3.8,0,0) node[right] {$u$};
  \draw[->,gray!80] (0,-0.85,0)--(0,1.2,0) node[below left] {$x_1$};
  \draw[->,gray!80] (0,0,-0.7)--(0,0,1.55) node[above] {$x_2$};
  \draw[thick] (-1.5,0.17,0.10)--(2.8,0.686,0.96);
  \node[above right] at (2.8,0.686,0.96) {$\ell_{p,v}$};
  \fill (0,0.35,0.40) circle (1.7pt);
  \node[above left] at (0,0.35,0.40) {$(0,p)$};
  \fill (1.7,0.554,0.74) circle (1.9pt);
  \node[above left,yshift=6pt] at (1.7,0.554,0.74) {$(u_0,p+u_0v)$};
  \draw[thick,dashed] (-1.5,-0.22,-0.525)--(2.8,-0.736,0.12);
  \node[below right] at (2.8,-0.736,0.12) {$\ell_{p',v'}$};
  \fill (1.7,-0.604,-0.045) circle (1.7pt);
\end{tikzpicture}
\caption{Two lines in Cayley coordinates. The plane $u=u_0$ meets
$\ell_{p,v}$ at $(u_0,P_{u_0}(p,v))=(u_0,p+u_0v)$.
The drawing is schematic.}
\label{fd:fig:cayley}
\end{figure}

The maps $P_u(p,v)=p+uv$ form a restricted family of projections from
$\R^4$ to $\R^2$. Related families in $\R^3$ were studied by
F\"assler and Orponen \cite{FO}. For line projections,
K\"aenm\"aki, Orponen and Venieri \cite{KOV} treated a model circle
of directions, Pramanik, Yang and Zahl \cite{PYZ} treated general
nondegenerate curves, and Gan, Guth and Maldague \cite{GGM} proved
exceptional-set estimates. Gan, Guo, Guth, Harris, Maldague and Wang
\cite{GGGHMW} treated plane projections. These results do not directly
apply to our family. Here we also use the product concentration
estimate \eqref{fd:eq:prism-bound}.

Our argument gives Haar nullity of the exceptional rotations, without
a quantitative bound for their Hausdorff dimension. Such a bound would
require further estimates, for example estimates averaged against
Frostman measures on the rotation group.
\section{Analytic estimates}
\label{sec:analytic-estimates}

We prove the estimates needed for the positive Lebesgue measure theorem.
We use the two-ends theorem to prove a dual estimate for a Frostman
measure. Refined decoupling gives the
corresponding Fourier estimate when each selected strip has small mass.
We also record the annular and angular multiplier estimates used to
combine these bounds.
\begin{lemma}\label{lem:frostman-annular}
Let $0<t\le2$ and let $\omega$ be a finite positive measure with
compact support in a fixed ball, satisfying
$\omega(B(x,r))\le C_\omega r^t$ for $x\in\R^2$, $0<r\le1$.
Then

\begin{equation}\label{eq:frostman-annular}
 \int_{R/2\le|\xi|\le2R}|\widehat\omega(\xi)|^2\dd\xi
 \lesssim R^{2-t}\omega(\R^2)
\end{equation}

for $R\ge1$. Equivalently,

\begin{equation}\label{eq:frostman-annular-polar}
 \int_{R/2}^{2R}rS_\omega(r)\dd r\lesssim R^{2-t}\omega(\R^2).
\end{equation}

\end{lemma}
\begin{proof}
Bounded support extends the Frostman bound to radii greater than
one after multiplying its constant by a fixed factor.
We choose a nonnegative $\Phi_{\rm ann}\in C_c^\infty(\R^2)$ that is at least one on
$B(0,2)$. Fourier inversion, rapid decay of $\check\Phi_{\rm ann}$, and the
Frostman bound give

\begin{align*}
 \int_{|\xi|\le2R}|\widehat\omega(\xi)|^2\dd\xi
 &\le \int\Phi_{\rm ann}(\xi/R)|\widehat\omega(\xi)|^2\dd\xi\\
 &\lesssim R^2\sum_{j\ge0}2^{-jN}
   \int\omega(B(x,2^jR^{-1}))\dd\omega(x)
 \lesssim R^{2-t}\omega(\R^2),
\end{align*}

where $N>t$ is fixed. This proves \eqref{eq:frostman-annular}, and
polar coordinates give \eqref{eq:frostman-annular-polar}.
This completes the proof.
\end{proof}

We shall also use the following consequence of the Frostman condition.
If $F$ has Fourier support in $B(0,CR)$, then

\begin{equation}\label{eq:trace}
 \|F\|_{L^2(\omega)}^2\lesssim R^{2-t}\|F\|_2^2.
\end{equation}

A nonnegative even rapidly decreasing majorant $\Phi_R(x)=R^2\Phi(Rx)$
of a reproducing kernel satisfies $|F|^2\lesssim |F|^2*\Phi_R$.
The Frostman condition gives $\omega*\Phi_R\lesssim R^{2-t}$. Thus,
by Fubini's theorem,
\[
 \int|F|^2\,d\omega
 \lesssim\int |F(x)|^2(\omega*\Phi_R)(x)\,dx
 \lesssim R^{2-t}\|F\|_{2}^{2},
\]
which proves \eqref{eq:trace}. These bounds use the original Frostman
constant, so they also apply to restrictions without renormalization.

For $0<w<1/2$, we write $L_w=\log(e/w)$. A shading of a line $\ell$ is a union $Y(\ell)$ of $w$-balls
with centres on $\ell$. It satisfies the $(\alpha_{\rm te},\beta_{\rm te})$ two-ends condition if
\[
 |Y(\ell)\cap B(z,w^{\alpha_{\rm te}})|\le Cw^{\beta_{\rm te}}|Y(\ell)|
\]
for every centre $z$, with a fixed constant $C$.
We use the general two-ends theorem of Wang and Wu
\cite[Theorem~1.5]{WW}. It extends their earlier direction-separated
result \cite{WW2024} to Katz--Tao families. O'Regan, Wu and Yi
\cite{OWY} give a simpler proof and an extension to transversal curves.
For the connection with projections, see Orponen and Shmerkin
\cite{OS}; the sharp planar Furstenberg bound is due to Ren and Wang
\cite{RW}.

A \(w\)-separated family of line parameters is a
Katz--Tao \((w,t)\)-set if every parameter ball of radius \(r\ge w\)
contains at most \(C_{\rm KT}(r/w)^t\) parameters. In a bounded slope
chart, the parameter of the line \(h_{\rm du}=b_{\rm lin}+k_{\rm lin}k\) is \((k_{\rm lin},b_{\rm lin})\). For a union
\(Y(\ell)\) of \(w\)-balls centred on a bounded part of a line, define
\[
 \Gamma_{Y,2-t}=
 \sup_{\ell,\,w\le r\le1,\,z\in\ell}
 (w/r)^{2-t}w^{-2}|Y(\ell)\cap B(z,r)|.
\]
In \cite[Theorem~1.5]{WW}, the scale $\delta$ is our $w$, the
two-ends exponents $(\varepsilon_1,\varepsilon_2)$ are our
$(\alpha_{\rm te},\beta_{\rm te})$, and $\gamma_{Y,t^*}$ is our $\Gamma_{Y,2-t}$,
since $t^*=\min\{t,2-t\}=2-t$ for $1<t<2$.

\begin{theorem}[Wang--Wu]\label{thm:WW}
Let \(1<t<2\) and \(0<\beta_{\rm te}<\alpha_{\rm te}<1\). Suppose that the lines form a
Katz--Tao \((w,t)\)-set and that their \(w\)-ball shadings satisfy the
\((\alpha_{\rm te},\beta_{\rm te})\)-two-ends condition. Suppose also that each shading
has area at least \(c\lambda w\), for some \(\lambda>0\) and fixed \(c>0\),
and lies in a fixed bounded region.
Then, for every
\(\varepsilon>0\),
\begin{equation}\label{eq:WW}
 \left|\bigcup_\ell Y(\ell)\right|
 \gtrsim
 w^{\varepsilon+t\alpha_{\rm te}/2}\lambda^{1/2}w^{(t-1)/2}
 \Gamma_{Y,2-t}^{-1/2}\sum_\ell |Y(\ell)|.
\end{equation}
The implicit constant may depend on \(\varepsilon,\alpha_{\rm te},\beta_{\rm te}\), the
two-ends constant, the Katz--Tao constant, and the fixed bounded region.
\end{theorem}

\begin{remark}\label{rem:WW-conventions}
The two-ends condition in \cite[Definition~1.2]{WW} is stated for
$w\times w^{\alpha_{\rm te}}$ tubes along a line. Our ball condition implies that
condition, since such a tube is covered by a fixed number of balls of
radius $w^{\alpha_{\rm te}}$. Conversely, the part of a ball of that radius within
a $w$-neighbourhood of a line is covered by a fixed number of those tubes.
Thus, the two formulations differ only in their constants.

We shall also use \eqref{eq:WW} for unions of $Cw$-balls, where $C\ge1$
is fixed and the other hypotheses are imposed on these shadings.
Here is the reduction. The Katz--Tao bound gives a fixed upper bound for
the number of line parameters within distance $2Cw$ of any one parameter.
A finite colouring therefore partitions the lines into $Cw$-separated
families. One family retains a fixed proportion of $\sum\limits_\ell|Y(\ell)|$.
On that family we apply the theorem at scale $\widetilde w=Cw$.
A ball of radius $\widetilde w^{\alpha_{\rm te}}$ is covered by a fixed number of balls of radius $w^{\alpha_{\rm te}}$, so the two-ends condition persists. Moreover,
\[
 \widetilde\lambda\sim\lambda,
 \qquad \Gamma_{Y,2-t}^{(\widetilde w)}
 \lesssim_C\Gamma_{Y,2-t}^{(w)},
\]
where the superscript indicates the scale used in the definition.
The latter bound follows directly by restricting the supremum to
$r\ge Cw$ and comparing the powers of the two scales. Substitution in
\eqref{eq:WW} proves the claimed version. A fixed dilation places the
ambient region in the unit ball used in \cite{WW}. It changes areas,
line-parameter distances and all radii by fixed factors; radii between
one and the original ambient diameter are handled by a fixed covering.
The same argument therefore proves the stated bounded-region version.

\end{remark}

\begin{lemma}\label{lem:KT-selection}
Fix $C_{\rm sep}\ge1$. Let $0<t\le2$ and let $\omega$ be a finite positive measure in a fixed bounded square with
$\omega(B(x,r))\le C_F r^t$, where $C_F>0$. Let $E$ be measurable
with $\omega(E)>0$, and let $0<w<1$. There is a finite
\(C_{\rm sep}w\)-separated set \(\mathcal {X}\subset E\) with
$|\mathcal X|\gtrsim\omega(E)/(C_Fw^t)$ and
\begin{equation}\label{eq:KT-selection}
 |\mathcal X\cap B(z,r)|\lesssim (r/w)^t
 \qquad(w\le r\le1).
\end{equation}
The constants may depend on $t$, the ambient square, and $C_{\rm sep}$, but
not on $C_F$, $E$, or $w$.
\end{lemma}
\begin{proof}
Replacing \(w\) by a comparable dyadic length changes only the constants.
We shall select occupied squares subject to an upper bound in every
larger dyadic square. We use the measure to define a fractional flow,
then apply the integral max-flow theorem to obtain the selected squares.
Consider the finite dyadic tree from a fixed ambient square down to
squares of side \(w\), with edges directed away from the root. We give the edge
entering a square \(Q\) the capacity
\[
 b(Q)=\lfloor(\operatorname{side}(Q)/w)^t\rfloor.
\]
Let \(Q_0\) denote the root square. We join the source to \(Q_0\) by an
edge of capacity \(b(Q_0)\).
We join every terminal square meeting \(E\) in positive \(\omega\)-measure
to a sink by an edge of capacity one. If \(C\) is a sufficiently large
absolute constant, then
\[
 \frac{\omega(E\cap Q)}{CC_Fw^t}
 \le \tfrac12(\operatorname{side}(Q)/w)^t\le b(Q),
\]
and the left side is at most one for a terminal square. Additivity of
the measure makes these values a feasible flow of total value
\(\omega(E)/(CC_Fw^t)\). Since all capacities
are integers, the integral max-flow theorem \cite{FF} gives an integral flow of at
least this value. It therefore selects occupied terminal squares, at most
one from each terminal square and at most \(b(Q)\) below every dyadic
square \(Q\).

We choose one point of \(E\) in each selected terminal square. A ball of
radius \(r\ge w\) meets a bounded number of dyadic squares with side
between \(r\) and \(2r\), so the capacity bound gives
\eqref{eq:KT-selection}. Finally, we colour the terminal grid by the two
coordinate indices modulo a fixed integer larger than \(C_{\rm sep}+3\). A largest
colour class retains a fixed proportion of the points and has the required
separation. If \(\omega(E)/(C_Fw^t)<1\), the positive feasible flow still implies
that the integral maximum flow has value at least one, so the same argument
applies. This completes the proof.
\end{proof}

\begin{lemma}\label{lem:dual}
Let \(\omega\) be a compactly supported positive measure of mass at most one
with a \(t\)-Frostman upper bound, where
\(1<t<2\), and fix \(0<\beta_{\rm te}<\alpha_{\rm te}<1\).
In a fixed bounded slope chart, associate \(x=(x_1,x_2)\) with
\[
 \ell_x=\{(k,h_{\rm du}):h_{\rm du}=x_2-kx_1\}.
\]
Let \(x\mapsto D_x\) be a measurable family of unions of grid intervals
of length comparable to \(w\). For each nonempty \(D_x\), let
\(Y_x\) be the union of one \(Cw\)-ball centred on \(\ell_x\)
above each grid interval in \(D_x\). Suppose that these shadings satisfy the
\((\alpha_{\rm te},\beta_{\rm te})\)-two-ends
condition and lie in a measurable set \(S\). Then, for every
\(\varepsilon>0\),
\begin{equation}\label{eq:dual}
 \int|D_x|\dd\omega(x)
 \lesssim L_w
 w^{-(2\varepsilon+t\alpha_{\rm te})/(4-t)}|S|^{2/(4-t)}.
\end{equation}
The implicit constant may depend on the constants listed in
Theorem~\ref{thm:WW}, on the Frostman constant, and on the fixed charts,
but not on \(w\), \(S\), or the family \(D_x\).
\end{lemma}
\begin{proof}
We group the points according to the lengths of their direction sets.
In each group, Lemma~\ref{lem:KT-selection} selects a Katz--Tao family
of dual lines. We then apply Theorem~\ref{thm:WW} to bound the measure
of that group in terms of $|S|$.

We divide the points for which \(D_x\ne\varnothing\) into \(O(L_w)\)
measurable classes
\[
 E_\lambda=\{x:\lambda\le |D_x|<2\lambda\},
 \qquad cw\le\lambda\le C,
\]
and consider a class with \(\omega(E_\lambda)>0\). We apply
Lemma~\ref{lem:KT-selection} to this class and denote the selected set by
\(\mathcal {X}\). The map
\(x\mapsto(-x_1,x_2)\), which gives the dual-line parameter, is an
isometry. The selected lines therefore form a Katz--Tao \((w,t)\)-set,
with
\[
 |\mathcal X|\gtrsim \omega(E_\lambda)w^{-t}.
\]
The separation constant in Lemma~\ref{lem:KT-selection} is chosen large
enough for the \(Cw\)-ball shadings. Bounded overlap along a line gives
\(|Y_x|\asymp\lambda w\). Moreover, for \(w\le r\le1\),
\[
 |Y_x\cap B(z,r)|\lesssim w\min\{r,\lambda\}.
\]
Hence,
\begin{equation}\label{eq:gamma-lambda}
 \Gamma_{Y,2-t}\lesssim(\lambda/w)^{t-1}.
\end{equation}
To see this, the expression defining \(\Gamma_{Y,2-t}\) is at most
\[
 \begin{cases}
 C(r/w)^{t-1},&r\le\lambda,\\
 C(w/r)^{2-t}\lambda/w,&r\ge\lambda.
 \end{cases}
\]
Since \(\lambda\gtrsim w\), both bounds are
\(\lesssim(\lambda/w)^{t-1}\).

We apply Theorem~\ref{thm:WW} with the $Cw$-ball convention in
Remark~\ref{rem:WW-conventions}. Substituting \eqref{eq:gamma-lambda}
and the lower bound for $|\mathcal X|$ into \eqref{eq:WW} gives

\begin{align}
 |S|&\gtrsim
 w^{\varepsilon+t\alpha_{\rm te}/2}\lambda^{1/2}w^{(t-1)/2}
 (\lambda/w)^{-(t-1)/2}|\mathcal X|\lambda w\notag\\
 &\gtrsim w^{\varepsilon+t\alpha_{\rm te}/2}
 \omega(E_\lambda)\lambda^{(4-t)/2}.
 \label{eq:weighted-dual-key}
\end{align}

Raising both sides to the power $2/(4-t)$ and using
$\omega(E_\lambda)\le1$, we obtain
\begin{align*}
 \omega(E_\lambda)\lambda
 &\lesssim \omega(E_\lambda)^{(2-t)/(4-t)}
 w^{-(2\varepsilon+t\alpha_{\rm te})/(4-t)}|S|^{2/(4-t)}\\
 &\le w^{-(2\varepsilon+t\alpha_{\rm te})/(4-t)}|S|^{2/(4-t)}.
\end{align*}
Summing over the \(O(L_w)\) classes proves \eqref{eq:dual}.
This completes the proof.
\end{proof}

We fix a nonnegative smooth partition of unity
$\{\chi_j\}_{j\in\mathbb Z/N_{w}\mathbb Z}$ on the circle, with
$N_{w}\asymp w^{-1}$, obtained by translating and rescaling one function. Its
supports have length $O(w)$ and bounded overlap, and
$\|\partial_\theta^k\chi_j\|_\infty\le C_kw^{-k}$.
If \(I\) is a consecutive set of indices, we define
\(\chi_I=\sum\limits_{j\in I}\chi_j\).
All radial cutoff functions in this section are smooth and supported in a
fixed annulus $\{c<|\xi|<C\}$ after scaling by $R$.

\begin{lemma}\label{lem:sector}
Let $\chi_I$ be as above, and let
$\phi$ be any fixed smooth annular cutoff.
Then, uniformly in $R\ge1$ and $I$,
\begin{equation}\label{eq:sector}
 \norm{\bigl[\phi(|\xi|/R)\chi_I(\xi/|\xi|)\bigr]^\vee}_1
 \less L_w.
\end{equation}
For two such cutoffs the corresponding product kernel has $L^1$ norm at
most $CL_w^2$. If an angular symbol is supported in an $O(w)$-arc and has
derivatives bounded by $C_kw^{-k}$, its annular kernel has uniformly
bounded $L^1$ norm.
The implicit constants may depend on the annular cutoff.
\end{lemma}

\begin{proof}
We first rescale an arc of length $O(w)$ to obtain the uniform kernel
bound. For a longer arc, we estimate the two endpoint cutoffs separately.
Their kernels have a logarithmic $L^1$ bound.

By dilation, it suffices to take $R=1$. We may also assume that $w$
is sufficiently small; for $w$ bounded below the symbol has uniformly
bounded derivatives, and integration by parts gives the claim. First suppose that the angular
symbol is supported in an $O(w)$-arc. After a rotation, we rescale the second frequency coordinate by $w$.
The resulting symbol is supported in a fixed compact rectangle, with
the first frequency coordinate bounded away from zero, and all its derivatives are
bounded independently of $w$. Integration by parts and the change of
variables give a bounded $L^1$ norm for its inverse transform.

We prove the logarithmic bound by an explicit decomposition at the
endpoints. Divide the consecutive set $I$ into a fixed number of
consecutive blocks, each of angular length at most $\pi/8$. It suffices
to estimate one block. Blocks of length $O(w)$ were treated above. For
any other block $I'$, rotate its middle direction to angle zero.
The partition of unity shows that $\chi_{I'}$ is
one between its two endpoint regions and zero outside the block enlarged
by $O(w)$. Its derivative is supported in two disjoint intervals of
length $O(w)$, with endpoints $a,b\in[-\pi/8,\pi/8]$.
If $g_a$ and $g_b$ are the restrictions of
$\frac{d}{d\theta}\chi_{I'}(e_\theta)$ to these intervals, then
\[
 \int g_a=1,\qquad \int g_b=-1,
 \qquad
 \chi_{I'}(e_\theta)=\Psi_a(\theta)-\Psi_b(\theta),
\]
where
\[
 \Psi_a(\theta)=\int_{-\infty}^{\theta}g_a(v)\,dv,
 \qquad \Psi_b(\theta)=-\int_{-\infty}^{\theta}g_b(v)\,dv.
\]
Both steps are bounded, their derivatives of positive order are supported
within $O(w)$ of their endpoints, and
$|\Psi_a^{(k)}|+|\Psi_b^{(k)}|\le C_kw^{-k}$ for $k\ge1$.
Choose a fixed angular cutoff equal to one on $[-\pi/8,\pi/8]$ and
supported in $(-\pi/4,\pi/4)$. The support of the block lies in
$[-\pi/8,\pi/8]$ for sufficiently small $w$. Multiplication by this cutoff expresses
the block multiplier exactly as two localized endpoint multipliers.

For either endpoint, rotate it to angle zero. The angular cutoff and
the fixed annulus then give $c\le\xi_1\le C$. The step changes only
when $|\xi_2|\lesssim w$. Let $\Psi$ denote the endpoint function
$\Psi_a$ or $\Psi_b$ in the rotated coordinate. Differentiating
$\Psi(\arctan(\xi_2/\xi_1))$ shows, for every fixed $n\ge0$, that
the corresponding multiplier $\psi_{\rm end}$ satisfies
\[
 \|\partial_{\xi_1}^n\psi_{\rm end}\|_1\lesssim_n1,
 \qquad
 \|\partial_{\xi_1}^n\partial_{\xi_2}^k\psi_{\rm end}\|_1
 \lesssim_n w^{1-k}\quad(k=1,2).
\]
Each derivative in $\xi_1$ of the step supplies a factor $\xi_2/w$,
which is bounded on its transition region. Each derivative in $\xi_2$
costs $w^{-1}$ there,
and that region has area $O(w)$. Terms differentiated only on the fixed
cutoffs have bounded derivatives and support of bounded area.
Integration by parts in the two coordinates now gives
\[
 |\psi_{\rm end}^{\vee}(x_1,x_2)|
 \lesssim (1+|x_1|)^{-2}
 \min\{1,|x_2|^{-1},w^{-1}|x_2|^{-2}\}.
\]
The integral of the second factor is at most
\[
 2\left(1+\int_1^{1/w}\frac{dv}{v}
       +w^{-1}\int_{1/w}^{\infty}\frac{dv}{v^2}\right)
 \lesssim L_w.
\]
Summing the fixed number of pieces proves \eqref{eq:sector}.
For a product of two sector cutoffs, insert an additional smooth radial
cutoff equal to one on $\supp\phi$, factor the multiplier, and apply
the $L^1$ convolution inequality. This completes the proof.
\end{proof}

We shall use a wave-packet formulation of refined decoupling, based on the result of Guth, Iosevich, Ou and Wang
\cite[Corollary~4.3]{GIOW}. See also
Carbery, Li, Pang and Yung \cite[Theorems~6.2--6.3]{CLPY}
for related formulations. The estimate below is stated in
the normalization needed for our application, allowing a
small thickening of the Fourier support and using slightly
enlarged tubes to measure multiplicity. Such modifications
are natural in the wave-packet setting, and the auxiliary
parameters are chosen sufficiently small so that the
corresponding losses can be absorbed into the
$R^{\varepsilon_{\mathrm{dec}}}$ factor.

For a rectangular tube $T$, the notation $cT$ means
dilation about its centre by a factor $c$ in both
coordinates. In particular, the tubes used to count
multiplicity below are slightly larger than those on which
the corresponding wave packets are concentrated.

\begin{theorem}[Refined decoupling]\label{thm:RD}
For every $\varepsilon_{\rm dec}>0$, there is
$\delta_0(\varepsilon_{\rm dec})>0$ with the following property.
Let $R$ be sufficiently large, $r\in[R/2,2R]$, and
$0<\delta_{\rm freq}\le\delta_0(\varepsilon_{\rm dec})$.
Let $\mathcal W$ be a finite family of rectangular tubes in a fixed ball,
all of length one and of the same half-width
\[
 C_T R^{-1/2+\delta_{\rm freq}},
\]
where $C_T>0$ is fixed. In each fixed circle chart their directions
belong to a uniform angular lattice of mesh comparable to $r^{-1/2}$,
and tubes of each direction have
bounded overlap. For $T\in\mathcal W$, let $F_T$ be Schwartz with Fourier
support in a rectangular cap about the circle $\{\xi:|\xi|=r\}$, of tangential length
$O(R^{1/2})$ and normal thickness at most $r^{3\delta_{\rm freq}/4}$,
whose normal is parallel to $T$. Suppose that
\[
 \|F_T\|_{L^6(\R^2\setminus2.1T)}\le r^{-100}\|F_T\|_6,
 \qquad a\le\|F_T\|_6\le2a
\]
for some $a>0$ and every $T\in\mathcal W$.
Let $Y$ be a union of $r^{-1/2}$-squares in a fixed ball, and suppose
$1\le M\le|\mathcal W|$ and
\[
 |\{T\in\mathcal W:3T\cap60Q\ne\varnothing\}|\le M
 \quad\text{for every square }Q\subset Y.
\]
Then
\begin{equation}\label{eq:RD}
 \left\|\sum_{T\in\mathcal W}F_T\right\|_{L^6(Y)}^2
 \lesssim_{\varepsilon_{\rm dec},\delta_{\rm freq}}R^{\varepsilon_{\rm dec}}
 \left(\frac M{|\mathcal W|}\right)^{2/3}
 \sum_{T\in\mathcal W}\|F_T\|_6^2.
\end{equation}
The constant may also depend on $C_T$, the fixed balls, the circle charts
and the overlap bound.
\end{theorem}

The light-strip estimate below follows the good-part estimate in
\cite[Section~5]{GIOW}. We allow an arbitrary strip-mass bound $q$ and
smooth selected strips, and we verify that the resulting packets satisfy
Theorem~\ref{thm:RD}. The dependence on $q$ follows from the mass estimate
\eqref{eq:multiplicity-mass}.

For each coarse angular cell $j$, let $e_j$ be its centre direction.
We fix a nonnegative smooth transverse partition of unity
\[
 \zeta_{j,k}(x)=\zeta\bigl((x\cdot e_j^\perp-kw)/w\bigr),
 \qquad \sum_{k\in\mathbb Z}\zeta_{j,k}(x)=1,
\]
where $\zeta$ is supported in $[-2,2]$. Let $G_{j,k}$ be the strip
$|x\cdot e_j^\perp-kw|\le2w$.
Before selecting the strips, we fix a sufficiently large tube-width
constant $C_{\rm pkt}$, and then a sufficiently large enlargement
constant $C_*$. These constants depend only on the fixed cutoffs and
packet geometry. Their required properties are verified below.
The enlargement $C_*G_{j,k}$ is taken about the same central line.
For $w\le q\le1$, we choose any family of indices $(j,k)$ satisfying

\begin{equation}\label{eq:light-test}
 \omega(C_*G_{j,k})\le q,
\end{equation}

and set $\mathrm l_j=\sum\limits_{k:(j,k)\text{ chosen}}\zeta_{j,k}$.
Then $0\le \mathrm l_j\le1$ and $|\partial_x^\gamma \mathrm l_j|\le C_\gamma w^{-|\gamma|}$.

We decrease the function $\delta_0$ in Theorem~\ref{thm:RD}, if necessary,
so that the following estimate also holds.

\begin{lemma}\label{lem:light}
Let $1<t<2$, and let $\omega$ be a positive measure of mass at most one
in a fixed bounded ball, satisfying a $t$-Frostman upper bound.
Let $\mathrm l_j$ be selected using this same measure $\omega$ in
\eqref{eq:light-test}.
For every $\varepsilon_{\rm dec}>0$, when
$0<\delta_{\rm freq}\le\delta_0(\varepsilon_{\rm dec})$, $w=R^{-1/2+\delta_{\rm freq}}$, and $r\in[R/2,2R]$,

\begin{equation}\label{eq:light-operator}
 \left\|\sum_j \mathrm l_j(x)\int_{S^1}e^{2\pi irx\cdot e}f_j(e)\dd\sigma(e)\right\|_{L^2(\omega)}^2
 \lesssim_{\varepsilon_{\rm dec},\delta_{\rm freq}}R^{-t/3+\varepsilon_{\rm dec}}q^{2/3}
 \sum_j\|f_j\|_2^2
\end{equation}

for $f_j\in L^2(S^1)$ supported in $\supp\chi_j$.
The constant depends on the original Frostman constant. It is uniform
in the chosen patches and in subfamilies of angular cells.
\end{lemma}
\begin{proof}
We construct wave packets adapted to the selected strips and verify
the hypotheses of Theorem~\ref{thm:RD}. The strip-mass bound then
controls the smoothed measure on regions where many packets meet.
H\"older's inequality and refined decoupling cancel the multiplicity
factors and give \eqref{eq:light-operator}.

By homogeneity, we may normalize $\sum\limits_j\|f_j\|_2^2=1$.

After a fixed translation and dilation, we may assume that
\(\supp\omega\subset B(0,1/10)\); this changes only the implicit
constants.

We set \(\rho=r^{-1/2}\sim R^{-1/2}\). In the fixed circle charts
of Theorem~\ref{thm:RD}, we use a smooth partition into caps of angular
length \(O(\rho)\). Their centres belong to the lattices specified in
that theorem. Only a bounded number of the angular cells of length
\(O(w)\) meets any one of these finer caps.

On a fine cap centred at $e_\theta$, we use the parameter
$\theta+\rho u$, with $u$ in a bounded interval. We expand the input
in an orthonormal Fourier series on a slightly larger interval and
multiply by a fixed smooth function equal to one on the input support.
Let $c_{\theta,n}$ be the Fourier coefficients. We include
$c_{\theta,n}$ in the definition of the packet $\psi_{\theta,n}$, so the
contribution of the cap is $\sum\limits_n\psi_{\theta,n}$. These packets satisfy
\begin{align}
 \sum_{\theta,n}|c_{\theta,n}|^2&\less1,\label{eq:coeff}\\
 |\psi_{\theta,n}(x)|
 &\less_J |c_{\theta,n}|\rho^{1/2}
 \bigl(1+\dist(x,\ell_{\theta,n})/\rho\bigr)^{-J}
 \quad\text{for }x\in B(0,1),\label{eq:packet-point}\\
 \|\vartheta\psi_{\theta,n}\|_6^2
 &\less \rho^{4/3}|c_{\theta,n}|^2
 \less R^{-2/3}|c_{\theta,n}|^2.\label{eq:packet-six}
\end{align}
Here the parallel axes $\ell_{\theta,n}$ have spacing comparable to $\rho$
and direction $e_\theta$. The function $\vartheta$ is a fixed smooth cutoff
equal to one on $B(0,1/5)$ and supported in $B(0,1/2)$. To prove
\eqref{eq:packet-point}, we note that the phase after rescaling the cap has
the form
$n u+r\rho(x\cdot e_\theta^\perp)u+O(|x|u^2)$.
Its derivatives of order at least two are uniformly bounded. Integration
by parts gives the displayed transverse decay. Integrating its sixth power
over the support of $\vartheta$ gives \eqref{eq:packet-six}. Bounded chart
Jacobians and overlap affect only constants. These assertions hold for the
contribution of each coarse angular cell separately; the bounded overlap
gives \eqref{eq:coeff} for their combined collection.

We multiply each packet from coarse angular cell \(j\) by each \(\zeta_{j,k}\) for
which \((j,k)\) is chosen, and by \(\vartheta\). Since
\(|e_j-e_\theta|\lesssim w\), there are only \(O(1)\) transverse cutoffs
whose central lines lie within \(O(w)\) of \(\ell_{\theta,n}\) on
\(B(0,1/2)\). We retain these terms. For the other strip cutoffs,
\eqref{eq:packet-point} gives an error \(O(R^{-J})\) for any prescribed
\(J\), since \(w/\rho\sim R^{\delta_{\rm freq}}\). The same integration by parts allows us
to restrict to \(|n|\le R^2\), with error \(O(R^{-J})\) in supremum norm
for every prescribed \(J\).
In each error estimate below, we first fix the required power $J$
and then choose the integration-by-parts order sufficiently large.
This order may depend on $J/\delta_{\rm freq}$, but not on $R$.
There are at most \(CR^{5/2}\) remaining packets before
accounting for the bounded number of transverse cutoffs.

We also discard axes missing $B(0,3/4)$. On the support of
$\vartheta$, such an axis has a fixed positive separation, so
\eqref{eq:packet-point}, Cauchy--Schwarz in the coefficients, and summation
of the transverse tails give $O(R^{-J})$ for every prescribed $J$.
All remaining axes therefore meet the fixed ambient ball; their unit
segments and required dilations can be placed in $B(0,3)$.
Hence, the geometric containment hypothesis in Theorem~\ref{thm:RD} holds.

We associate each remaining product
$\vartheta\zeta_{j,k}\psi_{\theta,n}$ with a tube \(T\) about
\(\ell_{\theta,n}\) of length one and half-width \(C_{\rm pkt}w\),
centred at the point of the axis nearest the origin. The fine and coarse axes differ
by \(O(w)\) in the fixed ambient ball.
With these fixed choices, we have
\[
 3T\cap60Q\ne\varnothing
 \quad\Longrightarrow\quad
 Q+B(0,cw)\subset C_*G_{j,k},
\]
where \(c\) is the fixed enlargement needed for smoothing at scale
\(R^{-1}\). These choices are independent of the measure and the scale.
Every tube comes from a chosen index \((j,k)\), so the strip on the
right satisfies \eqref{eq:light-test}.

For a fixed fine direction, the axes initially have spacing $\rho$. We
split their indices modulo a sufficiently large fixed multiple of
$w/\rho\sim R^{\delta_{\rm freq}}$. The tubes of half-width $C_{\rm pkt}w$ then have bounded overlap within
each residue class. A further bounded colouring separates the $O(1)$
coarse angular cells and strip cutoffs. All resulting families meet the directional
and overlap conditions in \eqref{eq:RD}.

We next verify the Fourier support and decay required in
Theorem~\ref{thm:RD}. Before inserting
the spatial cutoffs, each packet is supported on its fine circle cap in
Fourier space. The transverse cutoff $\zeta_{j,k}$ has Fourier frequencies
along $e_j^\perp$. We truncate these frequencies at $R^{\delta_{\rm freq}/4}/w$.
Their component normal to the fine cap is
$O(w)R^{\delta_{\rm freq}/4}/w=O(R^{\delta_{\rm freq}/4})$, and their tangential component is
$O(R^{1/2-3\delta_{\rm freq}/4})$. We likewise truncate the Fourier transform of
$\vartheta$ at $R^{\delta_{\rm freq}/4}$. The discarded parts, and their needed
derivatives, are $O(R^{-J})|c_{\theta,n}|$ for every prescribed $J$, since
the original cutoffs are smooth.

If \(L_\zeta=R^{\delta_{\rm freq}/4}/w\), the curvature of the circle contributes the
additional normal error
\[
 O(L_\zeta r^{-1/2}+L_\zeta^2/r)
 =O(R^{-3\delta_{\rm freq}/4}+R^{-3\delta_{\rm freq}/2}),
\]
which is smaller than the allowed normal thickness.

For each tube \(T\), let \(P_T\) be a smooth Fourier multiplier equal to
one on the truncated support. Its support is a rectangle of tangential
length $O(R^{1/2})$ and normal thickness at most $r^{3\delta_{\rm freq}/4}/10$.
After rescaling the two coordinates, its kernel has uniformly bounded
$L^1$ norm. We define
\[
 F_T=P_T\bigl(\vartheta\zeta_{j,k}\psi_{\theta,n}\bigr).
\]
The function inside the multiplier is smooth and compactly supported,
so $F_T$ is Schwartz. By definition, it has the exact Fourier support
required in \eqref{eq:RD}, and
\begin{equation}\label{eq:packet-six-sum}
 \sum_T\|F_T\|_6^2\less R^{-2/3}.
\end{equation}
After summing all $O(R^{5/2})$ projected packets, the projection and
truncation errors on \(\supp\omega\) are still $O(R^{-J})$ in supremum
norm, upon increasing the integration-by-parts orders. The projection kernel has
rapid decay outside transverse distance $O(\rho)$ and longitudinal distance
$O(R^{-3\delta_{\rm freq}/4})$. Together with \eqref{eq:packet-point}, this gives an
absolute $L^6$ tail $O(R^{-J})|c_{\theta,n}|$ outside $2.1T$. There is a
fixed transverse margin: the strip cutoff is supported in a smaller
transverse dilation of \(T\), while \eqref{eq:packet-point} controls the
spatial tails of the original wave packet. In the longitudinal direction,
$\supp\vartheta\subset B(0,1/2)$ lies a fixed positive distance from
the ends of $2.1T$, by our choice of the centre of $T$.
Thus, after summing over the charts,

\begin{equation}\label{eq:packet-decomposition-error}
 \vartheta\sum_j \mathrm l_jE_rf_j=\sum_TF_T+\mathcal E_J,
 \qquad
 \|\mathcal E_J\|_{L^\infty(\supp\omega)}\less_J R^{-J},
\end{equation}

provided the integration-by-parts orders are sufficiently large.

We discard the packets with \(\|F_T\|_6<R^{-100}\). Their Fourier
supports lie in \(B(0,CR)\), so Bernstein's inequality gives
\(\|F_T\|_\infty\lesssim R^{1/3}\|F_T\|_6\). We choose the projection and
tail errors above to be \(O(R^{-400})|c_{\theta,n}|\).

Since \(\sum\limits_{\theta,n}|c_{\theta,n}|^2\lesssim1\), we have
\(|c_{\theta,n}|\lesssim1\). For the retained packets,
\[
 \frac{\|F_T\|_{L^6(\R^2\setminus2.1T)}}{\|F_T\|_6}
 \lesssim R^{-300}\le r^{-100}
\]
for sufficiently large $R$. This is the relative decay required in
Theorem~\ref{thm:RD}.

The packet count
and Bernstein's inequality show that the discarded contribution is
\(O(R^{-50})\) on \(\supp\omega\); the implicit constant may depend on
\(\delta_{\rm freq}\). By \eqref{eq:packet-six-sum},
\(\|F_T\|_6\lesssim1\). We divide the remaining packets into
\(O(\log R)\) classes in which the \(L^6\) norms lie in one
interval \([2^k,2^{k+1})\). Thus, they differ by at most a factor of two. Each class now
satisfies the Fourier support and decay hypotheses of
Theorem~\ref{thm:RD}.

We now estimate the mass of the smoothed measure where tubes overlap.
We fix one of these classes, in one chart and one residue class, and
denote its family of tubes by $\mathcal W$. Let $F=\sum\limits_{T\in\mathcal W}F_{T}$.
The reproducing
inequality used in \eqref{eq:trace} gives
\[
 \|F\|_{L^2(\omega)}^2\less\int|F(x)|^2v_R(x)\dd x,
 \qquad v_R=\omega*\Phi_R,
 \quad \|v_R\|_\infty\less R^{2-t},\quad\|v_R\|_1\less1.
\]
Outside the fixed ambient ball, the $L^{3/2}$ norm of $v_R$ is
$O(R^{-J})$ for every fixed $J$, while $\|F\|_6$ has a polynomial bound
by \eqref{eq:packet-six-sum} and the packet count. H\"older's inequality
therefore bounds this contribution by $O(R^{-50})$. We tile that ball by
$\rho$-squares and write
\[
 n(Q)=|\{T:3T\cap60Q\ne\varnothing\}|.
\]
We group the squares according to the dyadic ranges
$M\le n(Q)<2M$ and denote their union by $Y_M$.
Theorem~\ref{thm:RD} applies to $Y_M$ with multiplicity parameter
$\min\{2M,|\mathcal W|\}$, which changes its bound by at most a fixed factor.
The containment proved above, together with the strip bound, gives

\begin{equation}\label{eq:multiplicity-mass}
 \int_{Y_M}v_R\less\frac{|\mathcal W|q}{M}
 +O\!\left(\frac{|\mathcal W|}{M}R^{-J}\right).
\end{equation}

For every square in $Y_M$, the count $n(Q)$ is at least $M$. Hence,
\[
 M\int_{Y_{M}}v_{R}
 \le\sum_{T}\sum_{\substack{Q\subset Y_M\\3T\cap60Q\ne\varnothing}}
        \int_{Q}v_{R}.
\]
For each tube, the union of these squares and its fixed
$w$-neighbourhood lie in $C_*G_{j,k}$. Smoothing at scale
$R^{-1}\ll w$ bounds the corresponding integral by $Cq+O(R^{-J})$,
using \eqref{eq:light-test}. Thus, the last sum is at most
$C|\mathcal W|q+O(|\mathcal W|R^{-J})$, which proves \eqref{eq:multiplicity-mass}.
Since $|\mathcal W|\lesssim R^{5/2}$ and $M\ge1$, we can make its error
$O(R^{-200})$ by increasing $J$.

H\"older's inequality, \eqref{eq:RD}, and \eqref{eq:multiplicity-mass} now give

\begin{align*}
 \int_{Y_M}|F|^2v_R
 &\le \|F\|_{L^6(Y_M)}^2
        \|v_R\|_\infty^{1/3}\left(\int_{Y_M}v_R\right)^{2/3}\\
 &\less R^{\varepsilon_{\rm aux}}R^{(2-t)/3}
     \left(\frac M{|\mathcal W|}\right)^{2/3}
     \left(\frac{|\mathcal W|q}{M}\right)^{2/3}
     \sum_T\|F_T\|_6^2+O(R^{-50})\\
 &\less R^{\varepsilon_{\rm aux}+(2-t)/3}q^{2/3}
     \sum_T\|F_T\|_6^2+O(R^{-50}).
\end{align*}

Squares with \(n(Q)=0\) lie outside every $2.1T$, so the decay estimate controls their contribution. The sum over the $O(\log R)$ values of $M$ costs a logarithmic factor. The charts, residue classes, and packet-norm ranges partition the packets into at most $R^{C\delta_{\rm freq}}(\log R)^C$ families. We apply Cauchy--Schwarz to their sum and use \eqref{eq:packet-six-sum} for the total packet energy. Together with \eqref{eq:packet-decomposition-error}, this gives
\[
 \left\|\sum_{j} \mathrm l_{j}E_{r}f_{j}\right\|_{L^2(\omega)}^2
 \less R^{-t/3+\varepsilon_{\rm aux}+C\delta_{\rm freq}}(\log R)^Cq^{2/3}+O(R^{-50}).
\]
We first choose \(\varepsilon_{\rm aux}<\varepsilon_{\rm dec}/4\), and then decrease $\delta_0(\varepsilon_{\rm dec})$ so that the packet estimate \eqref{eq:RD} holds with loss $\varepsilon_{\rm aux}$ and \(C\delta_0(\varepsilon_{\rm dec})<\varepsilon_{\rm dec}/4\). The argument holds
for every \(0<\delta_{\rm freq}\le\delta_0(\varepsilon_{\rm dec})\). The remaining power of $R^{\varepsilon_{\rm dec}}$
absorbs the fixed logarithmic factors. Since $q\ge w$ and $1<t<2$,
the error $O(R^{-50})$ is absorbed as well. This proves
\eqref{eq:light-operator}. The proof is unchanged for a subfamily of
coarse angular cells. This completes the proof.
\end{proof}

\section{Angular decomposition}\label{sec:angular-decomposition}

We shall decompose each Frostman measure into three angular parts.
The light and remaining heavy parts satisfy circular estimates.
The clustered parts will be estimated together after averaging over
rotations.

\subsection{Outline of the proof}\label{sec:proof-outline}

We first explain how these estimates prove
Theorem~\ref{thm:breaking-four-thirds}. Let \(1<t<2\) be a common
Frostman exponent for the two measures. In this outline we write \(w=R^{-1/2}\) and
omit logarithms and arbitrarily small power losses. In the proof we
use \(w=R^{-1/2+\delta_{\rm freq}}\), write out those losses, and
choose all parameters before summing in \(R\).

Let \(\omega\) be either of the two Frostman measures. For each
angular cell of length \(w\), divide the support into transverse
strips of width \(w\). Fix a mass threshold \(h=w^{1-b_0}\), with
\(0<b_0<1\). A strip is heavy if its remaining mass exceeds \(h\).
The difficulty is that the heavy directions through a point may
concentrate in a short interval. A two-ends estimate cannot be
applied to such a concentration.

We remove these concentrations by a finite procedure. At each point,
the removed directions lie in at most \(K\) intervals of length
\(O(w^{\alpha_c})\), where
\(K\lesssim w^{-\beta_c}(\log(e/w))^2\) and
\(0<2\beta_c<\alpha_c\). We denote their symbol by \(c\), the
remaining light symbol by \(\mathrm l\), and the remaining heavy symbol
by \(d\). Thus, \(c+\mathrm l+d=1\). Removing directions changes the strip
masses. We therefore repeat the angular and mass tests until both
are satisfied, and we do this simultaneously at every dyadic mass
threshold above \(h\). Figure~\ref{fig:angular-parts} illustrates
the three resulting parts at one point.

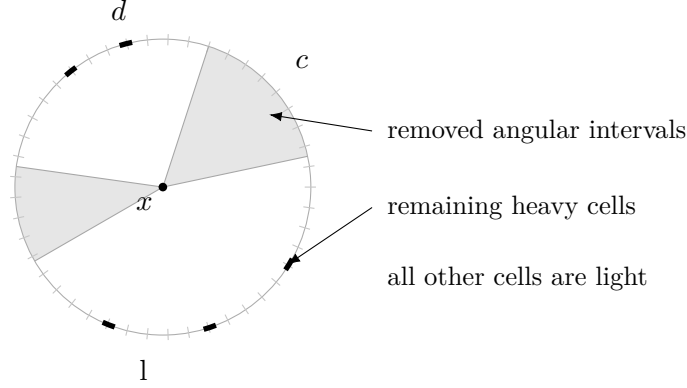
\begin{figure}[htbp]
\centering
\begin{tikzpicture}[scale=1.35,>=Latex]
 \fill[gray!20] (0,0) -- (12:1.45) arc[start angle=12,end angle=72,radius=1.45] -- cycle;
 \fill[gray!20] (0,0) -- (172:1.45) arc[start angle=172,end angle=210,radius=1.45] -- cycle;
 \draw[gray!65] (0,0) circle (1.45);
 \foreach \a in {0,8,...,352}{\draw[gray!45] (\a:1.39)--(\a:1.50);}
 \foreach \a in {102,126,246,286,326}{
   \draw[line width=2pt] (\a:1.45) arc[start angle=\a,end angle=\a+5,radius=1.45];
 }
 \draw[gray!70] (0,0)--(12:1.45) (0,0)--(72:1.45);
 \draw[gray!70] (0,0)--(172:1.45) (0,0)--(210:1.45);
 \fill (0,0) circle (1.2pt);
 \node[below left] at (0,0) {$x$};
 \node at (42:1.83) {$c$};
 \node at (104:1.79) {$d$};
 \node at (264:1.80) {$\mathrm l$};
 \draw[->] (2.05,.55)--(1.04,.71);
 \node[anchor=west,align=left,font=\small] at (2.10,.55)
   {removed angular intervals};
 \draw[->] (2.05,-.20)--(1.22,-.76);
 \node[anchor=west,align=left,font=\small] at (2.10,-.20)
   {remaining heavy cells};
 \node[anchor=west,align=left,font=\small] at (2.10,-.90)
   {all other cells are light};
\end{tikzpicture}
\caption{The angular partition at a fixed point. The shaded sectors
represent \(c\); the thick arcs represent \(d\). The picture shows the underlying cell classification. The symbols
use overlapping smooth cutoffs, and the remaining heavy part also
contains fixed neighbors of heavy bins. The cells have length comparable to
\(w\), and the removed intervals have length \(O(w^{\alpha_c})\).}
\label{fig:angular-parts}
\end{figure}

The light part satisfies the circular estimate
\[
 S_\omega^{\mathrm l}(r)\lesssim R^{-t/3}h^{2/3},\qquad r\asymp R,
\]
up to the losses just described. This follows from refined decoupling,
with the strip mass \(h\) left as a parameter.
For the remaining heavy strips, the dual lines satisfy a two-ends
condition. The Furstenberg estimate gives a tail bound
\[
 w\sum_{j,k:a_{jk}>q}a_{jk}\lesssim(q/w)^{-p_t},
 \qquad p_t=\frac2{2-t},\qquad q\ge h.
\]
The circular energy of the strips in one dyadic mass range is bounded
by \(w\sum a_{jk}^2\). Summing the square roots of these bounds over
the mass ranges is a geometric sum, since \(p_t>1\). Hence,
\[
 S_\omega^d(r)\lesssim h(h/w)^{-p_t}
 =w^{1+(p_t-1)b_0}.
\]

The clustered parts are treated together. For fixed source points,
two sectors of lengths \(\lambda_1,\lambda_2\) overlap after a
rotation only for an angular set of measure
\(O(\lambda_1+\lambda_2+w)\). The inverse Fourier kernel of each
sector pair has \(L^1\) norm \(O((\log(e/w))^2)\). The total length of the selected intervals is
\(O(Kw^{\alpha_c})\). Summing over at most \(K^2\) pairs therefore
gives a bound tending to zero because
\(\alpha_c>2\beta_c\). This averaging in \(L^1\) is the reason
we can retain the clustered parts in the measure.

The four other products have one light or remaining heavy factor.
Plancherel's theorem and rotational averaging pair its circular
estimate with an annular bound \(R^{2-t}\) for the other factor.
Thus, before losses, the two squared \(L^2\) bounds are
\[
 R^{2-t}R^{-t/3}h^{2/3}
   =w^{(8t-10-2b_0)/3},\qquad
 R^{2-t}w^{1+(p_t-1)b_0}
   =w^{2t-3+(p_t-1)b_0}.
\]
They decay exactly when
\[
 \frac{(3-2t)(2-t)}{t}<b_0<4t-5.
\]
This interval is nonempty if and only if \(t^2+t-3>0\), or
\(t>\sbase=(\sqrt{13}-1)/2\). Thus, the light and remaining heavy
estimates together give the threshold in Theorem~\ref{thm:breaking-four-thirds}.
To reach \(5/4\), one would need to improve at least one of these estimates.
The positive margins absorb the small power losses in the estimates below.
The resulting \(L^1\) bounds for the frequency pieces are summable,
which proves absolute continuity.

\subsection{Construction of the decomposition}

Throughout this section, let $1<t<2$. On a bounded support, a measure
with a larger Frostman exponent also satisfies the bound with exponent $t$.
We apply the following constructions separately to $\mu$ and $\nu$,
writing $\omega$ for either measure. We may translate the two supports
and apply a common dilation so that both lie in a fixed small ball.
These operations translate and dilate each rotated difference, so they
preserve absolute continuity.
All restrictions of $\omega$ retain its original Frostman upper bound;
we do not normalize their masses.

For dyadic $R\ge2$, let
\begin{equation}\label{imp:scale-parameters}
 \begin{gathered}
 w=R^{-1/2+\delta_{\rm freq}},\qquad p_t=\frac2{2-t},\qquad
 H=w^{-b_0},\qquad h=wH,\\
 A_{\rm l}(q)=R^{-t/3}q^{2/3},\qquad A_{\rm d}(q)=q(q/w)^{-p_t},
 \end{gathered}
\end{equation}
where $0<b_0<1$ and $0<\delta_{\rm freq}<1/4$. All auxiliary
parameters are fixed independently of $R$. We retain their losses in
the estimates below and choose them simultaneously in
Section~\ref{sec:proof-subcritical}.

Let $J_j$ be the disjoint angular cells of length comparable to $w$
associated with the smooth partition $\chi_j$ of
Section~\ref{sec:analytic-estimates}, and let
$e_j$ be their central directions. For an interval $I$ made up of cells,
$j\in I$ means $J_j\subset I$. Define the disjoint transverse bins
\begin{equation}\label{imp:ownership-bins}
 Q_{jk}=\{x:kw\le x\cdot e_j^\perp<(k+1)w\}.
\end{equation}
We use a fixed finite collection of bounded slope charts, with endpoints
on the angular grid. In each chart the change from angle to slope is
bi-Lipschitz. Tests in a chart always use the part of the angular set
in that chart. Thus, mass in another chart does not affect the test.
The smooth angular cutoffs may extend beyond the chart endpoints.
Their overlap remains bounded.

Let $0<2\beta_c<\alpha_c<1$. Fix a constant $C_{\rm 2e}>1$,
independent of $R$, for the angular test below.
We use all thresholds $h_n=2^nh<1$,
where $n$ is a nonnegative integer. The higher thresholds are needed
to estimate the whole distribution of residual bin masses: a
two-ends condition on one set of directions need not hold on a subset.
Let $k_j(x)$ be the index of the bin containing $x$.
For each chart, fix an ordered list of intervals of consecutive cells
with between $\lceil C_{\rm test}w^{\alpha_c}/w\rceil$ and
$2\lceil C_{\rm test}w^{\alpha_c}/w\rceil$ cells, clipped at its endpoints.
Choose the fixed constant $C_{\rm test}$ so that every interval of length
$w^{\alpha_c}$ in the slope coordinate lies in a test interval.
Order all pairs $(n,I)$ once and for all. The construction is as follows.
\begin{enumerate}
\item Set
\[
 c_j(x)=0,\qquad
 \mathcal P_{n,j}=\{k:\omega(Q_{jk})>h_n\}.
\]
At each $x$, start with an empty list of recorded angular intervals.
For the current coefficients and bin lists, set
\[
 D_{n,x}=\bigcup\{J_j:c_j(x)=0,\ k_j(x)\in\mathcal P_{n,j}\}.
\]
\item Keep the bin lists fixed. At each $x$, if a pair $(n,I)$ violates
\begin{equation}\label{imp:terminal-test}
 |D_{n,x}\cap I|\le C_{\rm 2e}^{-1}w^{\beta_c}|D_{n,x}|,
\end{equation}
where both occurrences of $D_{n,x}$ refer to its restriction to the
chart containing $I$, choose the first such pair, record $I$ with its
threshold $n$, and set $c_j(x)=1$ on every cell of $I$.
Recompute the sets $D_{n,x}$ and repeat until all angular tests hold.
\item Compute the residual bin masses
\[
 a_{jk}=\int_{Q_{jk}}(1-c_j(x))\dd\omega(x)
\]
and remove $k$ from $\mathcal P_{n,j}$ whenever $a_{jk}\le h_n$.
If a list changes, return to step~2; otherwise stop.
\end{enumerate}
Throughout the construction, the coefficients $c_j$ only increase,
the masses $a_{jk}$ and the lists $\mathcal P_{n,j}$ only decrease,
and $\mathcal P_{n+1,j}\subset\mathcal P_{n,j}$.
The selected cells at $x$ are exactly the union of its recorded
intervals. Each angular operation removes at least one cell from the
set that failed its test. The finite ordered choices therefore give
measurable coefficients.

\begin{lemma}
\label{imp:terminal}
The construction terminates after finitely many rounds, all final
coefficients are measurable, and for every $n,j,k$,
\[
 k\in\mathcal P_{n,j}\quad\Longleftrightarrow\quad a_{jk}>h_n.
\]
In each slope chart, every final set $D_{n,x}$ satisfies the
 two-ends condition at scale $w^{\alpha_c}$ with exponent $\beta_c$.
 Its fixed $O(w)$ thickening, and the associated $Cw$-ball shading
 in the dual plane, satisfy the corresponding condition with exponent
 $\beta_c/2$ for sufficiently small $w$. Here $C$ is any fixed
 enlargement constant.
At each $x$, the final selected set $\{j:c_j(x)=1\}$ is the union of at most
\[
 K\lesssim w^{-\beta_c}L_w^2
\]
recorded intervals of length $O(w^{\alpha_c})$.
\end{lemma}

\begin{proof}
We first verify termination and identify the final bin lists.
We then count the recorded intervals and check the two-ends condition.
There are finitely many entries $(n,j,k)$ in the lists. Each round
either ends the construction or removes an entry. Thus, the construction
terminates. The operation at each point is finite and
measurable as explained above, and every residual mass is an integral
of a measurable function.

The masses $a_{jk}$ decrease throughout. If a bin is absent from a list
at any stage, its mass at that stage is at most the corresponding
threshold. Its final mass is no larger. Conversely, the final test
removes every bin whose mass is at most its threshold. This proves the
stated equivalence.

For a fixed $x$, chart, and threshold $n$, the set $D_{n,x}$ only
shrinks during the entire construction. Each angular interval
recorded because of that threshold deletes more than a
$C_{\rm 2e}^{-1}w^{\beta_c}$ fraction of its current length.
Other angular deletions and bin updates can only reduce this length.
The initial length is bounded, and any nonempty set of cells has
length at least a fixed multiple of $w$. Thus, the number of such
records is at most
\[
 1+\frac{\log(C/w)}{-\log(1-C_{\rm 2e}^{-1}w^{\beta_c})}
 \lesssim w^{-\beta_c}L_w.
\]
There are $O(L_w)$ thresholds and a fixed number of charts,
giving the asserted bound on $K$.

At the final round every set defined by the lists passes the angular
test, and the lists do not change thereafter. Thus, all the final sets
satisfy the two-ends condition.
For a fixed $O(w)$ thickening, the cells contributing to its
intersection with an interval of length $O(w^{\alpha_c})$ lie in an
$O(w)$ enlargement of that interval. Within the chart, this enlargement
is covered by a bounded number of the test intervals. Each original
cell contributes only a bounded number of neighboring cells. Thus,
the final tests bound the thickened intersection by $Cw^{\beta_c}$
times the original total length, and hence by the same factor times
the thickened total length. The passage to areas of ball shadings
changes only fixed constants. For sufficiently small $w$, these
constants are absorbed by replacing $\beta_c$ with $\beta_c/2$.
This completes the proof.
\end{proof}

For the final coefficients, set $\omega_j=(1-c_j)\omega$ and
\[
 M(q)=w\sum_{j,k:a_{jk}>q}a_{jk},\qquad h\le q\le1.
\]
At a dyadic threshold $q=h_n$, the final list consists precisely of the
bins with $a_{jk}>h_n$. Hence,
\[
 \int|D_{n,x}|\dd\omega(x)
 =\sum_{j}|J_{j}|\sum_{k:a_{jk}>h_{n}}a_{jk}
 \sim M(h_{n}),
\]
since $|J_{j}|\sim w$.

\begin{lemma}\label{imp:terminal-tail}
For every $\varepsilon_F>0$ and $h\le q\le1$,
\begin{equation}\label{imp:terminal-tail-explicit}
 M(q)\lesssim
 L_w^{(4-t)/(2-t)}w^{-(2\varepsilon_F+t\alpha_c)/(2-t)}(q/w)^{-p_t}.
\end{equation}
The constant may depend on $t,\alpha_c,\beta_c,\varepsilon_F$ and the
original Frostman constant, but not on $R$ or $q$.
\end{lemma}
\begin{proof}
We estimate the area occupied by the dual rectangles in terms of
$M(q)$. The two-ends estimate then gives an inequality for $M(q)$.
We first take $q=h_n$ and fix one slope chart. Let $M_{\rm ch}(h_n)$
be its contribution to $M(h_n)$. The occupied dual rectangles have
area at most
\[
 Cw^2|\{(j,k):j\text{ belongs to the chart},\ a_{jk}>h_n\}|
 \le C\frac{w}{h_n}M_{\rm ch}(h_n).
\]
To see this, let $s_j$ be the central slope. Up to a fixed sign,
$x\cdot e_j^\perp=(x_2-s_jx_1)/\sqrt{1+(s_j)^2}$, so a transverse bin
fixes the dual height at $s_j$ to an interval of length $O(w)$.
Varying the slope across $J_j$ changes the height by at most
$O(w)|x_1|=O(w)$. One fixed enlargement therefore contains the
$Cw$-ball shading associated with each occupied bin.
By Lemma~\ref{imp:terminal}, these shadings satisfy the two-ends
condition. Lemma~\ref{lem:dual} gives
\[
 M_{\rm ch}(h_n)\lesssim L_w
 w^{-(2\varepsilon_F+t\alpha_c)/(4-t)}
 \left(\frac{w}{h_n}M_{\rm ch}(h_n)\right)^{2/(4-t)}.
\]
If $M_{\rm ch}(h_n)>0$, we divide by its power on the right and raise
to $(4-t)/(2-t)$. This gives
\[
 M_{\rm ch}(h_n)\lesssim
 L_w^{(4-t)/(2-t)}w^{-(2\varepsilon_F+t\alpha_c)/(2-t)}
 (h_n/w)^{-p_t}.
\]
The estimate is immediate when $M_{\rm ch}(h_n)=0$. We sum over the
fixed number of charts. For an intermediate $q$, we choose the largest
$h_n\le q$. Then $q<2h_n$, and monotonicity gives the same bound.
This completes the proof.
\end{proof}
\begin{lemma}
\label{imp:cap-bin-energy}
Let $\tau$ be a positive measure supported in the fixed ambient ball,
and set $a_k=\tau(Q_{jk})$. Uniformly for $r\sim R$,
\[
 \int_{S^1}|\chi_j(e)\widehat\tau(re)|^2\dd\sigma(e)
       \lesssim w\sum_k a_k^2.
\]
No spatial regularity of a density defining $\tau$ is required.
\end{lemma}

\begin{proof}
Expanding the square gives
\[
 \int_{S^1}|\chi_j(e)\widehat\tau(re)|^2\dd\sigma(e)
 =\iint K_{j}(x-y)\dd\tau(x)\dd\tau(y),
\]
where
\[
 K_j(x-y)=\int_{S^1}|\chi_j(e)|^2
                     e^{-2\pi i r(x-y)\cdot e}\dd\sigma(e).
\]
Bounded support and repeated angular integration by parts give
\[
 |K_j(z)|\lesssim_N
 w(1+|z\cdot e_j^\perp|/w)^{-N}.
\]
For transverse separation $O(w)$, the bound follows from the length
of the cap. Outside a sufficiently large fixed multiple of $w$, the phase
derivative is comparable to $R|z\cdot e_j^\perp|$. Smoothness of the
cap and $Rw^2\geq1$ give the displayed decay.
For $x\in Q_{jk}$ and $y\in Q_{j,k'}$ this is bounded by
$C_Nw(1+|k-k'|)^{-N}$. Thus, the left side is at most
\[
 C_N w\sum_{k,k'}(1+|k-k'|)^{-N}a_k a_{k'}
 \lesssim w\sum_k a_k^2,
\]
by the $\ell^2$ convolution inequality for a summable sequence.
This completes the proof.
\end{proof}

We now use the transverse patches $\zeta_{j,k}$ and their fixed enlarged strips from Lemma~\ref{lem:light}. Let $\mathrm l_j$ be the sum of patches whose enlarged strips meet no bin of $\omega_j$-mass greater than $h$.
Each enlarged strip has width $4C_*w$, so it meets at most
$4C_*+2$ bins and has $\omega_j$-mass at most $(4C_*+2)h$. Set
\begin{equation}\label{imp:fixed-symbols}
 \begin{split}
 c(x,e)&=\sum_j\chi_j(e)c_j(x),\\
 \mathrm l(x,e)&=\sum_j\chi_j(e)(1-c_j(x))\mathrm l_j(x),\\
 d(x,e)&=\sum_j\chi_j(e)(1-c_j(x))(1-\mathrm l_j(x)).
 \end{split}
\end{equation}
Then $c+\mathrm l+d=1$, and all three symbols are nonnegative.
For a bounded symbol $b(x,e)$, define
\begin{equation}\label{imp:symbol-transform}
 \mathcal F_\omega^b(\xi)=
 \int e^{-2\pi ix\cdot\xi}b(x,\xi/|\xi|)\dd\omega(x),\qquad
 S_\omega^b(r)=\int_{S^1}|\mathcal F_\omega^b(re)|^2\dd\sigma(e).
\end{equation}
For the two measures, define $T_R(b,b')$ by
\begin{equation}\label{imp:bilinear-output}
 \widehat{T_R(b,b')}(g,\xi)=\phi(|\xi|/R)
 \mathcal F_\mu^b(g^{-1}\xi)\overline{\mathcal F_\nu^{b'}(\xi)}.
\end{equation}
The Fourier transform here is in the planar variable $z$.
Norms of $T_R$ below are taken with respect to $\dd m(g)\dd z$ on
$\OO\times\R^2$.

\begin{lemma}\label{imp:fixed-source-bounds}
For every $\varepsilon_{\rm dec}>0$, let
$0<\delta_{\rm freq}\le\delta_0(\varepsilon_{\rm dec})$ be as in
Lemma~\ref{lem:light}. Uniformly for $r\asymp R$,
\begin{equation}\label{imp:light}
 S_\omega^{\mathrm l}(r)\lesssim
 R^{\varepsilon_{\rm dec}}w^{-\beta_c}L_w^4A_{\rm l}(h).
\end{equation}
For $b=1,c$,
\begin{equation}\label{imp:ann}
 \int_{R/2}^{2R}rS_\omega^b(r)\dd r
 \lesssim w^{-\beta_c}L_w^4R^{2-t}.
\end{equation}
For $b=1$, the factor $w^{-\beta_c}L_w^4$ can be omitted.
The constants may depend on the fixed auxiliary parameters and the
original Frostman constant, but not on $R$.
\end{lemma}
\begin{proof}
We first apply the light estimate to one fixed positive restriction
of $\omega$. We then partition the selected angular sets into dyadic
intervals, so that this estimate applies to each resulting restriction.
Let $I$ be an angular interval made up of cells, and let
$\omega_0\le\omega$ be a fixed positive measure. Suppose that, for
$j\in I$, each strip selected in $\mathrm l_j$ has enlarged-strip
$\omega_0$-mass at most $(4C_*+2)h$. Apply Lemma~\ref{lem:light}
to these cells with $q=(4C_*+2)h$, which lies in $[w,1]$ for
sufficiently large $R$. The Frostman constant is unchanged.
To apply duality, define
\[
 \mathcal E_{\rm sel}f(x)=\sum_{j\in I}\mathrm l_j(x)
    \int_{S^1}e^{2\pi irx\cdot e}\chi_j(e)f(e)\dd\sigma(e).
\]
The lemma and bounded angular overlap give
\[
 \|\mathcal E_{\rm sel}\|_{L^2(\sigma)\to L^2(\omega_0)}^2
 \lesssim R^{-t/3+\varepsilon_{\rm dec}}h^{2/3}.
\]
Here we used $\sum\limits_j\|\chi_jf\|_2^2\lesssim\|f\|_2^2$.
Fubini's theorem identifies the adjoint as
\[
 \mathcal E_{\rm sel}^*F(e)=\sum_{j\in I}\chi_j(e)
       \int e^{-2\pi irx\cdot e}\mathrm l_j(x)F(x)\dd\omega_0(x).
\]
Thus, $\|\mathcal E_{\rm sel}^*1\|_2^2\le\|\mathcal E_{\rm sel}\|^2
\|1\|_{L^2(\omega_0)}^2$ gives
\begin{equation}\label{imp:positive-source-light}
 \left\|\sum_{j\in I}\chi_j(e)\int e^{-2\pi irx\cdot e}
                  \mathrm l_j(x)\dd\omega_0(x)\right\|_{L^2(\sigma)}^2
 \lesssim R^{-t/3+\varepsilon_{\rm dec}}h^{2/3}\omega_0(\R^2).
\end{equation}
No spatial regularity of the density defining $\omega_0$ is needed.

At each point $x$, we cut the cyclic angular grid at a fixed index and
partition the consecutive blocks with $c_j(x)=0$ into canonical dyadic
intervals. At any given length there are at most $2(K+1)$ such intervals.
For each dyadic interval $I$, let $\xi_I(x)$ indicate that this interval
is selected, and set $\omega_I=\xi_I\omega$. The choices are measurable,
and the partition gives
\[
 1-c_{j}(x)=\sum_{I:j\in I}\xi_I(x).
\]
Therefore,
\[
 \mathcal F_{\omega}^{\mathrm l}(re)
 =\sum_{I}\sum_{j\in I}\chi_j(e)
   \int e^{-2\pi irx\cdot e}\mathrm l_{j}(x)\dd\omega_{I}(x).
\]
For $j\in I$, we have $\omega_I\le\omega_j$. Thus, each patch in
$\mathrm l_j$ has enlarged-strip $\omega_I$-mass at most $(4C_*+2)h$.
Apply \eqref{imp:positive-source-light} to the cells $j\in I$ with
this one fixed measure. The squared norm of their contribution is at most
$CR^{-t/3+\varepsilon_{\rm dec}}h^{2/3}\omega_I(\R^2)$.

At one dyadic length the angular supports have bounded overlap, and
\[
 \sum_{I\text{ of this length}}\omega_I(\R^2)\le2(K+1).
\]
There are $O(L_w)$ lengths. Cauchy--Schwarz across these lengths gives
\[
 S_\omega^{\mathrm l}(r)\lesssim
 (K+1)L_w^2R^{-t/3+\varepsilon_{\rm dec}}h^{2/3}.
\]
Since $K+1\lesssim w^{-\beta_c}L_w^2$, this proves
\eqref{imp:light}.

For $c$, we decompose the blocks with $c_j(x)=1$ in the same way.
By Lemma~\ref{lem:frostman-annular}, each resulting positive restriction
has annular energy bounded by $CR^{2-t}$ times its mass.
Bounded overlap at each length
and Cauchy--Schwarz across lengths give
\[
 \int_{R/2}^{2R}rS_\omega^c(r)\dd r
 \lesssim (K+1)L_w^2R^{2-t}.
\]
For $b=1$, Lemma~\ref{lem:frostman-annular} gives the estimate
without this factor. Substituting the bound for $K+1$ proves
\eqref{imp:ann}.
This completes the proof.
\end{proof}
\begin{lemma}\label{imp:cluster-pair}
Suppose that, at each source point, the symbol $b_1$ is a sum of
sector cutoffs associated with at most $K_\mu$ angular intervals,
whose total length is at most $\Lambda_\mu$. Suppose that $b_2$
has the corresponding bounds $K_\nu,\Lambda_\nu$.
The intervals are unions of the fixed angular cells and may depend
measurably on the source point. Then
\begin{equation}\label{imp:cluster-general}
 \|T_R(b_1,b_2)\|_1
 \lesssim L_w^2\bigl(K_\nu\Lambda_\mu
             +K_\mu\Lambda_\nu+K_\mu K_\nu w\bigr).
\end{equation}
In particular, common bounds $K,\Lambda$ give
$\|T_R(b_1,b_2)\|_1\lesssim(K\Lambda+K^2w)L_w^2$.
For the clustered symbols in \eqref{imp:fixed-symbols},
\begin{equation}\label{imp:cc}
 \|T_R(c,c)\|_{L^1(\dd m(g)\dd z)}
 \lesssim K^2w^{\alpha_c}L_w^2
 \lesssim w^{\alpha_c-2\beta_c}L_w^6.
\end{equation}
\end{lemma}
\begin{proof}
Fix source points $x,y$. For a pair of their angular intervals $I,J$,
the product of the corresponding cutoffs after a random
rotation is identically zero unless their smooth supports meet. This
occurs with probability at most $C(|I|+|J|+w)$, whereas the nonzero
annular kernel has $L^1$ norm at most $CL_w^2$ by
Lemma~\ref{lem:sector}. Thus, the integral in $g$ of the $L^1$ norm
of the kernel for this pair is at most
$CL_w^2(|I|+|J|+w)$. Summing gives
\[
 CL_w^2\left(K_\nu\sum_I|I|+K_\mu\sum_J|J|
                      +K_\mu K_\nu w\right)
 \le CL_w^2(K_\nu\Lambda_\mu+K_\mu\Lambda_\nu+K_\mu K_\nu w).
\]
Integration over the probability measures in $x,y$, and translation
invariance of the kernel norm, prove \eqref{imp:cluster-general}.
For positive finite measures the right side is multiplied by
the product of their masses. The same argument works on each
component of $\OO$.

For the clustered symbols, the selected cells form at most $K$ disjoint
consecutive blocks of total length at most $CKw^{\alpha_c}$, by
Lemma~\ref{imp:terminal}. Applying the common-bound estimate and
using $w\le w^{\alpha_c}$ gives \eqref{imp:cc}.
This completes the proof.
\end{proof}

The exponent $\alpha_c-2\beta_c$ is positive. Thus, the $L^1$
estimate for the clustered pair has the required decay, without a
pointwise Fourier estimate for either clustered part.

We keep the final coefficients $c_j$ and the fixed transverse partition.
Let $U_{q,j}$ be the
sum of patches whose required enlargements meet no $\omega_j$-bin of mass
greater than $q$. Let $n_h$ be the least integer with $2^{n_h}h\ge1$. For
$q=2^nh$, $0\le n<n_h$, define
\begin{equation}\label{imp:exact-layers}
 A_{q,j}=U_{2q,j}-U_{q,j},\qquad b_{q,j}=(1-c_j)A_{q,j}.
\end{equation}
The lists are nested, so $A_{q,j}$ is a nonnegative sum of newly selected
patches. Since $\omega_{j}(\R^2)\le1$ and $2^{n_h}h\ge1$, we have
$U_{2^{n_{h}}h,j}=1$. Thus,
\[
 \sum_{n=0}^{n_{h}-1}A_{2^{n}h,j}
 =U_{2^{n_{h}}h,j}-U_{h,j}=1-\mathrm l_{j}.
\]
This identity retains the contributions of neighboring inactive bins.
Every enlarged strip belonging to the layer has $\omega_j$-mass at most
$Cq$.

\begin{lemma}\label{imp:layer}
For each layer, at most $CM(q)/(wq)$ transverse bins meet its support,
and every such bin has $\omega_j$-mass at most $Cq$.
The estimates below hold uniformly for $r\asymp R$.
For $b_q=\sum\limits_j\chi_j b_{q,j}$ and every $\varepsilon_F>0$,
\begin{equation}\label{imp:layer-bound}
 S_{\omega}^{b_q}(r)\lesssim
 L_w^{(4-t)/(2-t)}w^{-(2\varepsilon_F+t\alpha_c)/(2-t)}A_{\rm d}(q).
\end{equation}
Hence, the full residual symbol satisfies
\begin{equation}\label{imp:base-residual}
 S_\omega^d(r)\lesssim
 L_w^{(4-t)/(2-t)}w^{-(2\varepsilon_F+t\alpha_c)/(2-t)}A_{\rm d}(h).
\end{equation}
For any positive measures $\tau_j\le b_{q,j}\omega$,
\begin{equation}\label{imp:schur}
 \int_{S^1}\left|\sum_j\chi_j(e)\widehat\tau_j(re)\right|^2\dd\sigma(e)
 \lesssim w\sum_{jk}\tau_j(Q_{jk})^2.
\end{equation}
\end{lemma}
\begin{proof}
We count the bins supporting one layer and use their upper mass bound
to estimate its circular energy. Minkowski's inequality then allows
us to sum over the layers.
A newly selected patch has an enlargement meeting a bin of mass $>q$
but none of mass $>2q$. Every bin in its support is within boundedly many
bin indices of such a bin. There are at most $M(q)/(wq)$ bins of mass $>q$,
and each has boundedly many neighbors. The enlarged strip contains
each support bin, giving its upper mass bound.
By bounded angular overlap and Lemma~\ref{imp:cap-bin-energy},
\begin{align*}
 \int_{S^1}\left|\sum_j\chi_j(e)\widehat\tau_j(re)\right|^2\dd\sigma(e)
 &\lesssim\sum_{j}\int_{S^{1}}
       |\chi_{j}(e)\widehat\tau_{j}(re)|^{2}\dd\sigma(e)\\
 &\lesssim w\sum_{j,k}\tau_j(Q_{jk})^2.
\end{align*}
This proves \eqref{imp:schur}. Applying it to
$\tau_j=b_{q,j}\omega$, we obtain
\[
 S_{\omega}^{b_{q}}(r)\lesssim w\frac{M(q)}{wq}(Cq)^2
 \lesssim qM(q).
\]
Lemma~\ref{imp:terminal-tail} proves \eqref{imp:layer-bound}.
The identity $d=\sum\limits_{n=0}^{n_h-1}b_{2^nh}$ and Minkowski's inequality
then give
\begin{align*}
 S_\omega^d(r)^{1/2}
 &\le\sum_{n=0}^{n_h-1}S_\omega^{b_{2^nh}}(r)^{1/2}\\
 &\lesssim
 L_w^{(4-t)/(2(2-t))}w^{-(2\varepsilon_F+t\alpha_c)/(2(2-t))}
 A_{\rm d}(h)^{1/2}\sum_{n\ge0}2^{n(1-p_t)/2}.
\end{align*}
The series converges because $p_t>1$, and
$A_{\rm d}(h)=h(h/w)^{-p_t}=wH^{1-p_t}$.
Squaring proves \eqref{imp:base-residual}. The constants have the
same dependencies as in Lemma~\ref{imp:terminal-tail}.
This completes the proof.
\end{proof}

We have now estimated all three fixed symbols. The clustered pair has
small average $L^1$ norm, while every other pair has a factor with a
circular estimate and a factor with an annular estimate.

\section{Proof of Theorem \ref{thm:breaking-four-thirds}}
\label{sec:proof-subcritical}

We shall prove Theorem~\ref{thm:breaking-four-thirds} by summing the
frequency estimates from the preceding section. We then obtain an
integrability bound for the densities and a sufficient condition with
two different Frostman exponents.

\begin{proof}[Proof of Theorem~\ref{thm:breaking-four-thirds}]
We choose the parameters so that the clustered pair has summable
$L^1$ norms and the other pairs have summable $L^2$ norms.
On a fixed ball, these bounds give an $L^1$ density for $\lambda_g$.
By decreasing the Frostman exponent, we may assume that
$\sbase<t<4/3$. A common dilation and separate translations place
both supports in the fixed ball used above. These operations change
each rotated difference measure only by a dilation and a translation,
so they preserve absolute continuity.

We first choose all parameters independently of $R$.
If we first set $w=R^{-1/2}$ and omit the auxiliary losses, the two
quantities to be summed are
\begin{equation}\label{imp:decay-exponents}
 R^{2-t}A_{\rm l}(h)=w^{(8t-10-2b_0)/3},\qquad
 R^{2-t}A_{\rm d}(h)=w^{2t-3+(p_t-1)b_0}.
\end{equation}
Both powers are positive precisely when
\begin{equation}\label{imp:threshold-choice}
 \frac{(3-2t)(2-t)}{t}<b_0<4t-5,
\end{equation}
since $p_t-1=t/(2-t)$. The difference between the upper and lower
endpoints is
\[
 (4t-5)-\frac{(3-2t)(2-t)}{t}
 =\frac{2(t^2+t-3)}{t}>0.
\]
Thus, $t>\sbase$ makes this interval nonempty. Its endpoints lie
between $0$ and $1$ because $\sbase<t<4/3$.
We make the following choices in order.
\begin{enumerate}
\item Choose $b_0$ satisfying \eqref{imp:threshold-choice}, and let
\[
 \Delta=\min\left\{\frac{4t-5-b_0}{3},
        \frac{2t-3+(p_t-1)b_0}{2}\right\}>0.
\]
These are the two decay exponents in $R$ before the auxiliary losses.
\item Fix the analytic loss and the angular parameters by
\[
 \varepsilon_{\rm dec}=\frac{\Delta}{16},\qquad
 \alpha_c=\min\left\{\frac14,
                  \frac{\Delta(2-t)}{64(1+t)}\right\},\qquad
 \beta_c=\frac{\alpha_c}{4},\qquad
 \varepsilon_F=\frac{\alpha_c}{2}.
\]
In particular, $0<2\beta_c<\alpha_c<1$ and
\[
 \beta_c\le\frac{\Delta}{256},\qquad
 \frac{2\varepsilon_F+t\alpha_c}{2-t}\le\frac{\Delta}{64}.
\]
\item Let $\delta_0(\varepsilon_{\rm dec})$ be supplied by
Lemma~\ref{lem:light}, and choose
\[
 0<\delta_{\rm freq}<
 \min\left\{\delta_0(\varepsilon_{\rm dec}),\frac14,
 \frac{\Delta}{16\left(1+(p_t-1)b_0+\frac23(1-b_0)\right)}\right\}.
\]
Use $w=R^{-1/2+\delta_{\rm freq}}$ and the definitions in
\eqref{imp:scale-parameters} at every dyadic scale $R$.
\end{enumerate}
All constants below may depend on these fixed choices and on the
original Frostman constants, but not on $R$.

The norms below are taken over $dm(g)\,dz$.
Plancherel's theorem, polar coordinates and invariance of Haar
measure give
\begin{equation}\label{imp:rotation-identity}
 \|T_R(b_1,b_2)\|_2^2
 =2\pi\int_0^\infty r|\phi(r/R)|^2
                 S_\mu^{b_1}(r)S_\nu^{b_2}(r)\,dr.
\end{equation}
Here the first and second symbols use $\mu$ and $\nu$, respectively.
The identity $c+\mathrm l+d=1$ implies
\[
 1-c_\mu c_\nu
 =\mathrm l_\mu+d_\mu+c_\mu\mathrm l_\nu+c_\mu d_\nu.
\]
Hence,
\[
 P_R\lambda_g=B_R(g,\cdot)+G_R(g,\cdot),
\]
where
\begin{align*}
 B_R&=T_R(c,c),\\
 G_R&=T_R(\mathrm l,1)+T_R(d,1)+T_R(c,\mathrm l)+T_R(c,d).
\end{align*}
In every term of $G_R$, one factor is $\mathrm l$ or $d$ and has a
circular estimate; the other is $1$ or $c$ and has an annular
estimate. Equations \eqref{imp:light}, \eqref{imp:ann},
\eqref{imp:base-residual}, and \eqref{imp:rotation-identity} therefore
show that
\begin{equation}\label{imp:output-bounds}
 \begin{split}
 \|B_R\|_1&\lesssim w^{\alpha_c-2\beta_c}L_w^6,\\
 \|G_R\|_2^2&\lesssim R^{2-t}\left[
 R^{\varepsilon_{\rm dec}}w^{-2\beta_c}L_w^8A_{\rm l}(h)
 +w^{-\beta_c-(2\varepsilon_F+t\alpha_c)/(2-t)}
 L_w^{(4-t)/(2-t)+4}A_{\rm d}(h)\right].
 \end{split}
\end{equation}
The first estimate is \eqref{imp:cc}.

For completeness, the two terms in the second estimate, with their
logarithms omitted, are $R^{-e_{\rm light}}$ and
$R^{-e_{\rm heavy}}$, where
\begin{align*}
 e_{\rm light}
 &=\frac{4t-5-b_0}{3}
   -\frac{2\delta_{\rm freq}}3(1-b_0)
   -\varepsilon_{\rm dec}-(1-2\delta_{\rm freq})\beta_c,\\
 e_{\rm heavy}
 &=\frac{2t-3+(p_t-1)b_0}{2}
   -\delta_{\rm freq}\bigl(1+(p_t-1)b_0\bigr)
   -\left(\frac12-\delta_{\rm freq}\right)
       \left(\beta_c+\frac{2\varepsilon_F+t\alpha_c}{2-t}\right).
\end{align*}
Our choices give $e_{\rm light},e_{\rm heavy}\ge3\Delta/4$.
For example, the loss in the first line is at most
$\Delta/16+\Delta/16+\Delta/256<\Delta/4$; that in the
second is at most
$\Delta/16+(\Delta/256+\Delta/64)/2<\Delta/4$.
The fixed logarithmic powers are at most $R^{\Delta/4}$ for large $R$.
Also, $\alpha_c-2\beta_c=\alpha_c/2$ and
$1/2-\delta_{\rm freq}>1/4$. Hence, after absorbing $L_w^6$,
\begin{equation}\label{imp:base-decay}
 \|B_R\|_1\lesssim R^{-\alpha_c/16},\qquad
 \|G_R\|_2^2\lesssim R^{-\Delta/2}.
\end{equation}
Thus, both estimates decay with fixed positive powers of $R$.

The measures $\lambda_g=(g_\#\mu)*\widetilde\nu$ are supported in
one fixed ball. Let $V$ be a larger ball whose boundary has positive
distance from that ball. The Schwartz kernel of the fixed annular
multiplier gives
\[
 \|P_R\lambda_g\|_{L^1(\R^2\setminus V)}=O_N(R^{-N}),
\]
uniformly in $g$, for every $N>0$. On $V$, the decomposition
and Cauchy--Schwarz give
\begin{equation}\label{imp:annular-L1-decay}
 \int_{\OO}\|P_R\lambda_g\|_1\,dm(g)
 \le\|B_R\|_1+|V|^{1/2}\|G_R\|_2+O_N(R^{-N})
 \lesssim R^{-\zeta_{\rm dec}}
\end{equation}
for some $\zeta_{{\rm dec}}>0$. This argument requires no support condition on
$B_R$ or $G_R$.

The right side of \eqref{imp:annular-L1-decay} is summable over
dyadic $R$. By Tonelli's theorem,
\[
 \sum_R\|P_R\lambda_g\|_{1}<\infty
\]
for almost all $g$. Hence, $\sum\limits_RP_R\lambda_g$ converges in $L^1$ for these $g$. The low-frequency part in \eqref{eq:dyadic-resolution} is also in $L^1$, since it is the
convolution of a finite measure with a Schwartz function. The
resolution of the identity identifies the resulting sum with
$\lambda_g$ as a distribution. Hence, $\lambda_g$ has an $L^1$
density for almost all $g$. The rotation identity holds on both
components of $\OO$.

For Borel sets $A,B$, we choose
\[
 \sbase<t<\min\{4/3,\dimH A,\dimH B\}.
\]
Frostman's lemma gives $t$-Frostman probabilities on compact subsets
of $A$ and $B$. For almost all $g$, their convolution is an
absolutely continuous probability supported on a compact subset
of $gA-B$. Thus, $\LL^2(gA-B)>0$. The difference set is analytic
and hence Lebesgue measurable. This completes the proof.
\end{proof}

The decay in \eqref{imp:annular-L1-decay} also gives higher
integrability of the densities. It controls the measure of the
rotated difference sets as follows.

\begin{corollary}\label{imp:integrability}
Under the measure hypotheses of
Theorem~\ref{thm:breaking-four-thirds}, there is $\eta_{\rm int}>0$ such that
the densities $f_g=d\lambda_g/d\LL^2$ satisfy
\[
 \int_{\OO}\int_{\R^2}f_g(z)^{1+\eta_{\rm int}}\,dz\,dm(g)<\infty.
\]
If $A',B'$ are compact sets supporting these probabilities, then
\[
 \int_{\OO}\LL^2(gA'-B')^{-\eta_{\rm int}}\,dm(g)<\infty.
\]
\end{corollary}

\begin{proof}
We interpolate the averaged $L^1$ estimate with the $L^\infty$
bound supplied by the Frostman condition. H\"older's inequality
will then give the assertion about the difference sets.
Let $t\in(\sbase,4/3)$ be the exponent chosen in the preceding proof.
Since $\nu$ is a probability, the Frostman bound for $\mu$ gives
\[
 \lambda_{g}(B(z,r))
 =\int\mu\bigl(B(g^{-1}(z+y),r)\bigr)\,d\nu(y)
 \le C_{\mu} r^{t}
\]
for $g\in\OO,z\in\R^{2},0<r\le1$.
The annular Schwartz kernel then gives
\[
 \|P_R\lambda_{g}\|_{L^\infty(dm(g)\,dz)}\lesssim R^{2-t}.
\]
Interpolating this with \eqref{imp:annular-L1-decay} yields
\[
 \|P_R\lambda_{g}\|_{L^{1+\eta_{\rm int}}(dm(g)\,dz)}
 \lesssim R^{[-\zeta_{\rm dec}+(2-t)\eta_{\rm int}]/(1+\eta_{\rm int})}.
\]
We choose $0<\eta_{\rm int}<\zeta_{\rm dec}/(2-t)$. Then the dyadic series
converges in $L^{1+\eta_{\rm int}}(dm(g)\,dz)$, and its sum, including the smooth
low-frequency term, is the density already identified in $L^1$.
This proves the first claim.

For almost all $g$, H\"older's inequality on $gA'-B'$ gives
\[
 1\le\LL^2(gA'-B')^{\eta_{\rm int}/(1+\eta_{\rm int})}\|f_g\|_{1+\eta_{\rm int}}.
\]
Raising to the power $1+\eta_{\rm int}$, we obtain
\[
 \LL^2(gA'-B')^{-\eta_{\rm int}}
 \le\|f_g\|_{1+\eta_{\rm int}}^{1+\eta_{\rm int}}
 =\int_{\R^{2}}f_{g}(z)^{1+\eta_{\rm int}}\,dz.
\]
The second claim follows by integration in $g$. This completes the proof.
\end{proof}

The same argument allows different Frostman exponents for the two
measures. The two mass thresholds can then be chosen independently.

\begin{corollary}\label{cor:mixed-exponents}
Let \(1<\alpha,\beta<2\), and suppose that
\begin{equation}\label{eq:mixed-region}
 \alpha^2+\alpha\beta-2\alpha+4\beta-6>0,
 \qquad
 \beta^2+\alpha\beta-2\beta+4\alpha-6>0.
\end{equation}
If \(\mu,\nu\) are compactly supported probability measures with
Frostman exponents \(\alpha,\beta\), respectively, then
\((g_\#\mu)*\widetilde\nu\ll\LL^2\) for almost all \(g\in\OO\).
Hence, for Borel sets \(A,B\) with dimensions \(a,b>1\),
\(\LL^2(gA-B)>0\) for almost all \(g\) if
\begin{equation}\label{eq:mixed-set-region}
 a^2+ab-2a+4b-6>0,
 \qquad b^2+ab-2b+4a-6>0.
\end{equation}
\end{corollary}

\begin{proof}
We choose the two mass thresholds so that all four $L^2$ estimates
decay. This gives the two inequalities in \eqref{eq:mixed-region}.
For the circular estimates of the light and remaining heavy parts
of \(\mu\), we use \(t=\alpha\) in Section~\ref{sec:angular-decomposition}
and pair these bounds with the annular estimate \(R^{2-\beta}\) for
\(\nu\). For the other two estimates, we interchange the measures
and use \(t=\beta\). We apply the angular construction separately
with the two Frostman exponents and thresholds
\[
 h_\mu=w^{1-b_\mu},\qquad h_\nu=w^{1-b_\nu},
 \qquad 0<b_\mu,b_\nu<1.
\]
The annular estimates for \(\mu\) and \(\nu\) are
\(R^{2-\alpha}\) and \(R^{2-\beta}\), respectively, with the
explicit additional losses for the clustered symbols as in
\eqref{imp:ann}. The circular estimates for the light and remaining
heavy parts are those of Section~\ref{sec:angular-decomposition},
with the corresponding exponent and threshold. Setting
\(w=R^{-1/2}\) and omitting the auxiliary losses, the four powers in the squared \(L^2\) estimates
for the non-clustered products are
\begin{align*}
 &\frac{2\alpha+6\beta-10-2b_\mu}{3},
 &&2\beta-3+\frac{\alpha}{2-\alpha}b_\mu,\\
 &\frac{2\beta+6\alpha-10-2b_\nu}{3},
 &&2\alpha-3+\frac{\beta}{2-\beta}b_\nu.
\end{align*}
For example, the first row comes from
\[
 R^{2-\beta}R^{-\alpha/3}h_\mu^{2/3},\qquad
 R^{2-\beta}h_\mu(h_\mu/w)^{-2/(2-\alpha)}.
\]
The first two powers are positive if and only if
\begin{equation}\label{eq:mixed-b-choice}
 \max\left\{0,\frac{(3-2\beta)(2-\alpha)}{\alpha}\right\}
 <b_\mu<\min\{1,\alpha+3\beta-5\}.
\end{equation}
The unrestricted lower endpoint is less than \(1\) for
\(1<\alpha,\beta<2\). If \(\beta<3/2\), it is positive;
if \(\beta\ge3/2\), the unrestricted upper endpoint is positive.
Thus, the interval in \eqref{eq:mixed-b-choice} is nonempty exactly
when
\[
 \alpha+3\beta-5
 -\frac{(3-2\beta)(2-\alpha)}{\alpha}
 =\frac{\alpha^2+\alpha\beta-2\alpha+4\beta-6}{\alpha}>0.
\]
The other two powers give the symmetric condition.

Fix \(b_\mu,b_\nu\) with strict margins in these intervals. There
are only four positive decay margins. Choose the two analytic losses,
the two Furstenberg losses and common angular parameters sufficiently
small compared with their minimum, with
\(0<2\beta_c<\alpha_c\). Then choose one positive
\(\delta_{\rm freq}\) small enough for both light estimates and
all four powers. The explicit losses in \eqref{imp:light},
\eqref{imp:ann} and \eqref{imp:base-residual} tend to zero with
these parameters. Their logarithmic factors are absorbed by the
remaining margins. Lemma~\ref{imp:cluster-pair} applies with the
two interval families separately, so its \(L^1\) estimate is
unchanged. The decomposition and summation in the proof of
Theorem~\ref{thm:breaking-four-thirds} therefore give absolute
continuity.

Finally, if \(a,b>1\) satisfy \eqref{eq:mixed-set-region}, choose
\(1<\alpha<a\), \(1<\beta<b\), both less than \(2\), sufficiently
close to \(a,b\) that \eqref{eq:mixed-region} still holds.
Frostman's lemma and the measure conclusion prove the assertion
for sets. This completes the proof.
\end{proof}

On the diagonal, \eqref{eq:mixed-set-region} reduces to
\(a^2+a-3>0\), as in Theorem~\ref{thm:breaking-four-thirds}.
The asymmetric region also contains \((a,b)=(1.29,1.34)\): the
left sides in \eqref{eq:mixed-set-region} are \(0.1727\) and
\(0.0042\), whereas neither inequality in
\eqref{eq:previous-area-region} holds.
Conversely, \((a,b)=(1.5,1.2)\) satisfies
\eqref{eq:previous-area-region}, since \(a+b/2=2.1>2\), but fails
the first inequality in \eqref{eq:mixed-set-region}, whose left-hand
side is \(-0.15\).

\FloatBarrier
\section{Sharpness}
\label{sec:sharpness}

The dimension bound in Theorem~\ref{thm:full-dimension} is sharp for
every pair of dimensions. For nonempty Borel sets $A,B\subset\R^2$,
the map $(a,b)\mapsto ga-b$ is Lipschitz. The standard product inequality
for Hausdorff and packing dimensions, see \cite{Mattila},
gives, for every $g\in\OO$,
\[
 \dimH(gA-B)
 \le\min\{2,\dimH(A\times B)\}
 \le\min\{2,\dimH A+\dimP B\}.
\]

For every $0<c<2$, let $K_c\subset[0,1]$ be the two-interval Cantor set
with contraction ratio $4^{-1/c}<1/2$. The four planar similitudes
defining $K_c\times K_c$ satisfy the open set condition, so this compact
set has Hausdorff and packing dimensions both equal to $c$.
For $c=0$ we use a point, and for $c=2$ we use the unit square.

Thus, for any $a,b\in[0,2]$ there are nonempty compact sets $A,B\subset\R^2$
with
\[
 \dimH A=\dimP A=a,\qquad \dimH B=\dimP B=b.
\]
For these sets the upper bound and Theorem~\ref{thm:full-dimension} give
\[
 \dimH(gA-B)=\min\{a+b,2\}
 \quad\text{for almost all }g\in\OO.
\]
In particular, the case $a+b=2$ also gives full Hausdorff dimension.
This sharpness statement concerns dimension and makes no claim about
positive Lebesgue measure at the endpoint. This proves the asserted sharpness.

\begin{remark}
The threshold in Theorem~\ref{thm:breaking-four-thirds} comes from
balancing the light and heavy estimates. One possible approach to
improving it is to localize the light estimate to shorter angular
intervals and to use more information about the distribution of the
remaining heavy directions. This would require estimates with a gain
that remains after the summation over intervals and mass thresholds. However the improvement may be very small. 

\end{remark}

\section{Intersections}
\label{sec:intersections}

The estimates used to prove Theorem~\ref{thm:breaking-four-thirds} also
give lower bounds for the dimensions of intersections. The passage from
an averaged $L^p$ estimate to an intersection theorem follows the
argument of Mattila \cite[Theorem~3.1]{Mattila2024}. For this application,
the estimate must be uniform for restrictions and rescalings of the
measures and must retain both total masses. We establish this form of
the estimate below and give the arguments in detail.
We use $T_g$ and $\lambda_g$ as in \eqref{eq:lambda-definition}.

\begin{theorem}\label{int:intersection-theorem}
Let $\alpha,\beta\in(\sbase,2]$, and let $\mu,\nu$ be compactly supported
Frostman probability measures on $\R^2$ with exponents $\alpha$ and
$\beta$, respectively. For almost all $g\in\OO$, the measure
$\lambda_g=(g_\#\mu)*\widetilde\nu$ is absolutely continuous and has
the following property. If $A,B\subset\R^2$ are Borel sets with
$\mu(A)=\nu(B)=1$, then
\begin{equation}\label{int:intersection-lower-bound}
 \dimH\bigl(gA\cap(B+z)\bigr)\ge\alpha+\beta-2
 \qquad\text{for $\lambda_g$-almost all }z.
\end{equation}
In particular, the translations satisfying
\eqref{int:intersection-lower-bound} contain a Lebesgue measurable set
of positive measure.
\end{theorem}

We first prove the required uniform estimate for finite measures.

\begin{proposition}\label{int:mass-integrability}
Let $\sbase<t<4/3$ and $C_0\ge1$. There are
$\eta_{\rm int}>0$ and $C<\infty$, depending only on $t$ and $C_0$,
with the following property. Let $\mu',\nu'$ be finite positive Borel
measures supported in $B(0,1)$ such that
\[
 \mu'(B(x,r))\le C_0r^t,
 \qquad \nu'(B(x,r))\le C_0r^t
 \quad(x\in\R^2,\ r>0).
\]
Then $\lambda'_g=(g_\#\mu')*\widetilde{\nu'}$ is absolutely continuous
for almost all $g\in\OO$. Its densities $f'_g$ satisfy
\begin{equation}\label{int:mass-integrability-bound}
 \int_{\OO}\int_{\R^2}|f'_g(z)|^{1+\eta_{\rm int}}\dd z\dd m(g)
 \le C\mu'(\R^2)\nu'(\R^2).
\end{equation}
The densities can be chosen jointly measurable in $(g,z)$.
\end{proposition}

\begin{proof}
We retain the total masses in the estimates of
Sections~\ref{sec:analytic-estimates}--\ref{sec:angular-decomposition},
and then interpolate the two parts separately.
Set
\[
 m_\mu=\mu'(\R^2),\qquad m_\nu=\nu'(\R^2).
\]
All norms below are over $\OO\times\R^2$ with measure
$\dd m(g)\dd z$.

Choose the parameters in the proof of
Theorem~\ref{thm:breaking-four-thirds}, depending only on $t$. In
particular, $\Delta>0$ and $0<\alpha_c\le1/4$ are as in that proof.
We use $w,h,p_t,A_{\rm l},A_{\rm d}$ from
\eqref{imp:scale-parameters}. In the estimates below, let $\omega$ denote either $\mu'$ or $\nu'$.
In particular, for every measurable set $E$, we use the restriction
$\omega|_E$ without dividing by its mass $\omega(E)$.

First consider Lemma~\ref{lem:dual}. The proof before the use of
$\omega(E_\lambda)\le1$ gives
\[
 \omega(E_\lambda)\lambda
 \lesssim \omega(E_\lambda)^{(2-t)/(4-t)}
 w^{-(2\varepsilon+t\alpha_{\rm te})/(4-t)}|S|^{2/(4-t)}.
\]
Use $\omega(E_\lambda)\le\omega(\R^2)$ and sum over the
$O(L_w)$ classes. Thus, the stronger form of \eqref{eq:dual} is
\begin{equation}\label{int:dual-mass}
 \int|D_x|\dd\omega(x)
 \lesssim \omega(\R^2)^{(2-t)/(4-t)}L_w
 w^{-(2\varepsilon+t\alpha_{\rm te})/(4-t)}|S|^{2/(4-t)}.
\end{equation}
Its constant is uniform for a fixed Frostman upper bound. In
particular, Lemma~\ref{lem:KT-selection}, used in its proof, requires
no positive lower bound on $\omega(E_\lambda)$.

In the proof of Lemma~\ref{imp:terminal-tail}, apply
\eqref{int:dual-mass} to the shadings in one chart. For a dyadic
threshold $q=h_n$, their union has area at most
$C(w/q)M_{\rm ch}(q)$. Therefore,
\[
 M_{\rm ch}(q)\lesssim
 \omega(\R^2)^{(2-t)/(4-t)}L_w
 w^{-(2\varepsilon_F+t\alpha_c)/(4-t)}
 \bigl((w/q)M_{\rm ch}(q)\bigr)^{2/(4-t)}.
\]
If $M_{\rm ch}(q)>0$, divide by its power on the right and raise to
$(4-t)/(2-t)$. The resulting power of $\omega(\R^2)$ is one. Summing
over the fixed charts and using monotonicity between dyadic
thresholds gives
\begin{equation}\label{int:tail-mass}
 M(q)\lesssim \omega(\R^2)
 L_w^{(4-t)/(2-t)}w^{-(2\varepsilon_F+t\alpha_c)/(2-t)}
 (q/w)^{-p_t},\qquad h\le q\le1.
\end{equation}
The estimate is immediate for a chart with $M_{\rm ch}(q)=0$.
Since $S_\omega^{b_q}(r)\lesssim qM(q)$, the layer summation in
Lemma~\ref{imp:layer} now gives
\begin{equation}\label{int:heavy-mass}
 S_\omega^d(r)\lesssim \omega(\R^2)
 L_w^{(4-t)/(2-t)}w^{-(2\varepsilon_F+t\alpha_c)/(2-t)}A_{\rm d}(h).
\end{equation}
This holds uniformly for $r\sim R$.

The light and annular estimates also retain one power of the mass.
Indeed, Lemma~\ref{lem:light} is an operator estimate with a constant
uniform over measures satisfying the fixed Frostman upper bound.
Its errors have already been absorbed into that estimate. Applying
the adjoint to $1\in L^2(\omega_0)$ therefore gives exactly the factor
$\omega_0(\R^2)$ in \eqref{imp:positive-source-light}, with no additive
error independent of the mass. At each dyadic angular length, the
restrictions in Lemma~\ref{imp:fixed-source-bounds} satisfy
\[
 \sum_{I\text{ of this length}}\omega_I(\R^2)
 \le2(K+1)\omega(\R^2).
\]
The same sum applies to the clustered annular estimate, and the
mass factor in Lemma~\ref{lem:frostman-annular} is explicit. Hence,
\begin{align}
 S_\omega^{\mathrm l}(r)
 &\lesssim \omega(\R^2)
 R^{\varepsilon_{\rm dec}}w^{-\beta_c}L_w^4A_{\rm l}(h),
 \label{int:light-mass}\\
 \int_{R/2}^{2R}rS_\omega^b(r)\dd r
 &\lesssim \omega(\R^2)w^{-\beta_c}L_w^4R^{2-t},
 \qquad b=1,c.
 \label{int:annular-mass}
\end{align}
For $b=1$, the factor $w^{-\beta_c}L_w^4$ can be omitted.
The angular construction and its bounds are unchanged for masses at
most one: the thresholds still stop at $1$, and the number of
recorded intervals at each point has the same bound.

Use the decomposition
$P_R\lambda'_g=B_R(g,\cdot)+G_R(g,\cdot)$ from the proof of
Theorem~\ref{thm:breaking-four-thirds}. The proof of
Lemma~\ref{imp:cluster-pair} integrates its kernel bound against the
two source measures. Consequently, its right side is multiplied by
$m_\mu m_\nu$. In each of the four terms defining $G_R$,
\eqref{imp:rotation-identity} pairs one circular estimate with one
annular estimate. Equations
\eqref{int:heavy-mass}--\eqref{int:annular-mass} supply one mass factor
from each source. The same parameter calculation as in
\eqref{imp:base-decay} therefore gives
\begin{equation}\label{int:tracked-decay}
 \|B_R\|_1\lesssim R^{-\alpha_c/16}m_\mu m_\nu,
 \qquad
 \|G_R\|_2^2\lesssim R^{-\Delta/2}m_\mu m_\nu.
\end{equation}
These estimates hold for $R\ge R_0$, where $R_0$ and the implicit
constants depend only on $t$, the fixed Frostman bound, and the fixed
cutoffs. None depends on either total mass or the particular
measures. The packet and angular parameters, and the scale cutoffs
in the preceding lemmas, have precisely these dependencies.

The fixed annular kernel and the Frostman bound also give, for every
$R\ge2$,
\begin{equation}\label{int:trivial-bounds}
 \|P_R\lambda'\|_1\lesssim m_\mu m_\nu,
 \qquad
 \|P_R\lambda'\|_\infty\lesssim R^{2-t},
 \qquad
 \|P_R\lambda'\|_2^2\lesssim R^{2-t}m_\mu m_\nu.
\end{equation}
For the middle estimate, use
$\lambda'_g(B(z,r))\lesssim r^t\min\{m_\mu,m_\nu\}\le Cr^t$.
The last estimate follows by multiplying the first two.
Thus, for $R\ge R_0$,
\[
 \|G_R\|_1\le\|P_R\lambda'\|_1+\|B_R\|_1
 \lesssim m_\mu m_\nu,
 \qquad
 \|B_R\|_2^2\lesssim R^{2-t}m_\mu m_\nu.
\]

Set $\eta_{\rm int}=\alpha_c/64$. For $0<\eta_{\rm int}<1$,
H\"older's inequality gives
\[
 \|F\|_{1+\eta_{\rm int}}^{1+\eta_{\rm int}}
 \le\|F\|_1^{1-\eta_{\rm int}}\|F\|_2^{2\eta_{\rm int}}.
\]
Apply this separately to $G_R$ and $B_R$. From
\eqref{int:tracked-decay} and \eqref{int:trivial-bounds},
\begin{align*}
 \|G_R\|_{1+\eta_{\rm int}}^{1+\eta_{\rm int}}
 &\lesssim R^{-\Delta\eta_{\rm int}/2}m_\mu m_\nu,\\
 \|B_R\|_{1+\eta_{\rm int}}^{1+\eta_{\rm int}}
 &\lesssim
 R^{-(1-\eta_{\rm int})\alpha_c/16+(2-t)\eta_{\rm int}}
 m_\mu m_\nu
 \lesssim R^{-\alpha_c/64}m_\mu m_\nu.
\end{align*}
The last inequality uses $2-t<1$ and $\eta_{\rm int}<1/2$.
Both powers decay. For the finitely many scales $2\le R<R_0$,
\eqref{int:trivial-bounds} gives directly
\[
 \|P_R\lambda'\|_{1+\eta_{\rm int}}^{1+\eta_{\rm int}}
 \le\|P_R\lambda'\|_\infty^{\eta_{\rm int}}
       \|P_R\lambda'\|_1
 \lesssim R^{(2-t)\eta_{\rm int}}m_\mu m_\nu.
\]
The low-frequency convolution $P_0\lambda'$ satisfies the same bound
without the power of $R$, by its fixed Schwartz kernel.

Minkowski's inequality now shows that the resolution
\eqref{eq:dyadic-resolution} converges in
$L^{1+\eta_{\rm int}}(\dd m(g)\dd z)$, with the
$(1+\eta_{\rm int})$-th power of its norm bounded by
$Cm_\mu m_\nu$. Each summand is jointly measurable. Moreover, the
sum of their $L^{1+\eta_{\rm int}}(\R^2)$ norms is finite for almost
every $g$, by Fubini's theorem and the finiteness of Haar measure.
For these $g$, the resolution identifies the sum with $\lambda'_g$
as a distribution. It is therefore its density. Restoring the fixed
normalizations proves \eqref{int:mass-integrability-bound} and the
joint measurability assertion. This completes the proof.
\end{proof}

\begin{proof}[Proof of Theorem~\ref{int:intersection-theorem}]
Choose
\[
 \sbase<t<\min\{\alpha,\beta,4/3\},
 \qquad q=1+\eta_{\rm int}>1,
\]
where $\eta_{\rm int}$ is supplied by
Proposition~\ref{int:mass-integrability} for the uniform Frostman constant
specified below. Write $\Pi=\mu\otimes\nu$. Enlarge the Frostman
constants to at least one, so that their bounds hold at all radii.
Since $\mu$ and $\nu$ are probability measures, these bounds also
imply $t$-Frostman bounds.
Thus Theorem~\ref{thm:breaking-four-thirds} gives
$\lambda_g\ll\LL^2$ for almost all $g$.

Now define a measure on
$\OO\times\R^2$ by
\[
 \Lambda(E)=\int_{\OO}\int_{\R^4}
  \mathbf 1_E(g,T_gp)\dd\Pi(p)\dd m(g).
\]
The absolute continuity just established and Fubini's theorem give
$\Lambda\ll m\times\LL^2$. For every dyadic product cube $D=Q\times Q'$,
define $\Lambda_D$ in the same way with $\Pi|_D$ in place of $\Pi$.
Since $\Lambda_D\le\Lambda$, choose nonnegative jointly Borel densities
$f(g,z)$ and $f^D(g,z)$ for these measures. We now explain why these jointly measurable densities give the
desired densities for almost every fixed rotation.
Fix a dyadic product cube $D$ and a dyadic square $Q\subset\R^2$.
By the definitions of $\Lambda$ and $\Lambda_D$, for every Borel
set $E\subset O(2)$,
\[
 \int_E\int_Q f(g,z)\,dz\,dm(g)
 =\int_E\lambda_g(Q)\,dm(g),
\]
and
\[
 \int_E\int_Q f^D(g,z)\,dz\,dm(g)
 =\int_E (T_g)_\#(\Pi|_D)(Q)\,dm(g).
\]
Since these identities hold for every Borel set $E$, the corresponding
integrals over $Q$ agree for almost every $g$.
There are only countably many dyadic squares $Q$, so we may choose
one full-measure set of rotations on which both identities hold
for every such $Q$.

Dyadic squares determine finite Borel measures on $\R^2$.
Consequently, for every rotation in this full-measure set,
$f(g,\cdot)$ and $f^D(g,\cdot)$ are densities of
$\lambda_g$ and $(T_g)_\#(\Pi|_D)$, respectively.
Finally, there are only countably many dyadic product cubes $D$.
Taking the intersection of the corresponding full-measure sets,
we obtain one full-measure set of rotations for which these
conclusions hold simultaneously for every $D$.
We write the resulting densities as $f_g$ and $f_g^D$.

Fix $k\ge1$, set $r=2^{-k}$, and let $\cD_k$ denote the half-open dyadic
squares of side $r$ in $\R^2$. For $Q\in\cD_k$ with lower-left corner
$x_Q$, set
\[
 \mu^Q(E)=r^{-\alpha}\mu\bigl(Q\cap(x_Q+rE)\bigr).
\]
Then $\supp\mu^Q\subset[0,1]^2$ and
\[
 \mu^Q(B(u,\rho))\le C_\mu\rho^\alpha
 \qquad(\rho>0).
\]
For $\rho\le1$ this is bounded by $C_\mu\rho^t$, and for $\rho>1$
the uniform bound on $\mu^Q(\R^2)$ gives the same conclusion with a
larger constant. Define $\nu^{Q'}$ similarly, using $\beta$ and the
lower-left corner $y_{Q'}$. After translating $[0,1]^2$ by
$(-1/2,-1/2)$, all these measures are supported in $B(0,1)$ and satisfy
a common $t$-Frostman bound. Its constant $C_0$ depends only on the
original Frostman constants and exponents, not on $k,Q,Q'$. Moreover,
\begin{equation}\label{int:rescaled-piece-masses}
 \sum_{Q\in\cD_k}\mu^Q(\R^2)=r^{-\alpha},
 \qquad
 \sum_{Q'\in\cD_k}\nu^{Q'}(\R^2)=r^{-\beta}.
\end{equation}

Let $F_g^{QQ'}$ be the density of
$(g_\#\mu^Q)*\widetilde{\nu^{Q'}}$. Translating either measure only
translates this density, so Proposition~\ref{int:mass-integrability}
applies without changing its $L^q$ norm. Since
\[
 f_g^{Q\times Q'}(z)
 =r^{\alpha+\beta-2}
 F_g^{QQ'}\bigl((z-gx_Q+y_{Q'})/r\bigr)
\]
almost everywhere, that proposition and
\eqref{int:rescaled-piece-masses} give
\begin{align}
 \sum_{Q,Q'\in\cD_k}\int_{\OO}\int_{\R^2}
       \bigl(f_g^{Q\times Q'}(z)\bigr)^q\dd z\dd m(g)
 &\le C r^{(\alpha+\beta-2)q+2}
       \sum_{Q,Q'}\mu^Q(\R^2)\nu^{Q'}(\R^2)\notag\\
 &\le C\,2^{-k(\alpha+\beta-2)\eta_{\rm int}}.
 \label{int:dyadic-density-sum}
\end{align}

For each $g$ in the common full-measure set, let
$\{\sigma_{g,z}\}_z$ be the probability measures obtained by disintegrating
$\Pi$ along $T_g$. Thus, for every Borel set $E\subset\R^4$,
\[
 \Pi(E)=\int\sigma_{g,z}(E)\dd\lambda_g(z),
 \qquad \sigma_{g,z}(T_g^{-1}(z))=1
 \quad\text{for $\lambda_g$-almost all }z.
\]
Uniqueness of Radon--Nikodym derivatives implies
\[
 \sigma_{g,z}(Q\times Q')
 =\frac{f_g^{Q\times Q'}(z)}{f_g(z)}
\]
for $\lambda_g$-almost all $z$, simultaneously for all $k,Q,Q'$.
Define the measurable function
\[
 S_k(g,z)=\sum_{Q,Q'\in\cD_k}
       \left(\frac{f_g^{Q\times Q'}(z)}{f_g(z)}\right)^q
\]
on $\{0<f_g<\infty\}$, and set it equal to zero elsewhere. It agrees
almost everywhere with
$\sum\limits_{Q,Q'}\sigma_{g,z}(Q\times Q')^q$.
For every $0<\gamma<\alpha+\beta-2$, Tonelli's theorem and
\eqref{int:dyadic-density-sum} yield
\[
 \int_{\OO}\int_{\R^2}
   \left(\sum_{k\ge1}2^{k\gamma\eta_{\rm int}}S_k(g,z)\right)
   f_g(z)^q\dd z\dd m(g)
 \le C\sum_{k\ge1}
       2^{-k(\alpha+\beta-2-\gamma)\eta_{\rm int}}<\infty.
\]
Consequently, for almost all $g$ and $\lambda_g$-almost all $z$,
\begin{equation}\label{int:fibre-summability}
 \sum_{k\ge1}2^{k\gamma\eta_{\rm int}}S_k(g,z)<\infty.
\end{equation}
Here we used that $f_g^q\dd z$ and $\lambda_g=f_g\dd z$ have the same
null sets on $\{0<f_g<\infty\}$, which has full $\lambda_g$-measure.
Take common full-measure sets for a fixed countable sequence
$0<\gamma_j\uparrow\alpha+\beta-2$.

Fix such $g,z$, and write $\tau=\sigma_{g,z}$. Let $D_k(p)$ denote the
half-open dyadic product cube of side $2^{-k}$ containing $p$. For one
of these exponents $\gamma$, let $\mathrm{Bad}_k$ be the union of the product
cubes $D=Q\times Q'$ at scale $2^{-k}$ for which
$\tau(D)>2^{-k\gamma}$. Since these half-open cubes are disjoint,
\[
 \tau(\mathrm{Bad}_k)
 =\sum_{D\subset\mathrm{Bad}_k}\tau(D)^q\tau(D)^{-\eta_{\rm int}}
 \le 2^{k\gamma\eta_{\rm int}}S_k(g,z).
\]
By \eqref{int:fibre-summability} and Borel--Cantelli, for
$\tau$-almost every $p$,
$\tau(D_k(p))\le2^{-k\gamma}$ for all sufficiently large $k$.
Intersecting these full-measure sets over $\gamma_j$ gives
\[
 \liminf_{k\to\infty}
 \frac{\log\tau(D_k(p))}{\log 2^{-k}}
 \ge\alpha+\beta-2
 \qquad\text{for $\tau$-almost every }p.
\]

For completeness, this dyadic estimate also gives the required
Hausdorff dimension statement without any regularity assumption on the
conditional measures. Let $F$ be Borel with $\tau(F)>0$, and fix one
of the positive exponents $\gamma_j$. There are a Borel set
$F_0\subset F$ with $\tau(F_0)>0$ and an integer $k_0$ such that
\[
 \tau(D_k(p))\le2^{-k\gamma_j}
 \qquad(p\in F_0,\ k\ge k_0).
\]
A set $U$ with diameter $d\in(2^{-k-1},2^{-k}]$ meets at most $16$
dyadic cubes of side $2^{-k}$ in $\R^4$. Each of those cubes that meets
$U\cap F_0$ has $\tau$-mass at most $2^{-k\gamma_j}$.
Thus every countable cover $\{U_i\}$ of $F_0$ by sets with
$0<d_i=\operatorname{diam}U_i<2^{-k_0}$ satisfies
\[
 \tau(F_0)\le16\sum_i(2d_i)^{\gamma_j}.
\]
Sets of diameter zero meeting $F_0$ have zero $\tau$-mass by the same bound,
so allowing them in a cover does not change the conclusion.
It follows that $\HH^{\gamma_j}(F_0)>0$, and hence
$\dimH F\ge\alpha+\beta-2$ by letting $j\to\infty$.

Finally, fix Borel sets $A,B$ with $\mu(A)=\nu(B)=1$.
Since $\Pi(A\times B)=1$, disintegration gives
$\sigma_{g,z}(A\times B)=1$ for $\lambda_g$-almost all $z$.
The conditional measure is also carried by $T_g^{-1}(z)$, so the result
just proved applies to $(A\times B)\cap T_g^{-1}(z)$. The map
\[
 u\longmapsto(g^{-1}u,u-z)
\]
is a similarity of ratio $\sqrt2$ from $gA\cap(B+z)$ onto this fibre
intersection. Therefore
$\dimH(gA\cap(B+z))\ge\alpha+\beta-2$ for $\lambda_g$-almost all $z$.
The admissible translations contain a measurable set of full
$\lambda_g$-measure, which has positive Lebesgue measure because
$\lambda_g\ll\LL^2$. This proves all the assertions of the theorem.
\end{proof}

\begin{corollary}\label{int:set-intersections}
Let $A,B\subset\R^2$ be Borel sets.
\begin{enumerate}[label=\textup{(\roman*)}]
\item Suppose that $\min\{\dimH A,\dimH B\}>\sbase$. There is a Borel
set $\mathcal G\subset\OO$ with $m(\mathcal G)=1$ such that, for every
$g\in\mathcal G$ and every $\varepsilon>0$, there is a Lebesgue measurable
set $Z_{g,\varepsilon}\subset\R^2$ of positive measure for which
\[
 \dimH\bigl(gA\cap(B+z)\bigr)
 \ge\dimH A+\dimH B-2-\varepsilon
 \qquad\text{for every }z\in Z_{g,\varepsilon}.
\]
\item Suppose that $\alpha,\beta\in(\sbase,2]$ and
$\HH^\alpha(A)>0$, $\HH^\beta(B)>0$. For almost all $g\in\OO$, there
is a Lebesgue measurable set $Z_g\subset\R^2$ of positive measure such that
\[
 \dimH\bigl(gA\cap(B+z)\bigr)\ge\alpha+\beta-2
 \qquad\text{for every }z\in Z_g.
\]
\end{enumerate}
\end{corollary}

\begin{proof}
For \textup{(i)}, write $a=\dimH A$ and $b=\dimH B$. Choose increasing
sequences $\alpha_j\uparrow a$ and $\beta_j\uparrow b$ with
$\sbase<\alpha_j<a$ and $\sbase<\beta_j<b$.
By Frostman's lemma, compact subsets $A_j\subset A$ and $B_j\subset B$
support probability measures with exponents $\alpha_j$ and $\beta_j$,
respectively. Apply Theorem~\ref{int:intersection-theorem} to each pair
of measures, and intersect the resulting countably many sets of
orthogonal transformations of full measure. Replacing this intersection
by a Borel subset of full measure gives $\mathcal G$.
For $g\in\mathcal G$ and $\varepsilon>0$, choose $j$ such that
$\alpha_j+\beta_j>a+b-\varepsilon$. The theorem gives a Lebesgue
measurable set of translations of positive measure on which
\[
 \dimH\bigl(gA\cap(B+z)\bigr)
 \ge\dimH\bigl(gA_j\cap(B_j+z)\bigr)
 \ge\alpha_j+\beta_j-2>a+b-2-\varepsilon.
\]
This proves \textup{(i)} with the same set $\mathcal G$ for every
$\varepsilon>0$.

For \textup{(ii)}, Frostman's lemma and compact approximation
\cite[Theorems~8.8 and~8.13]{Mattila} give compact subsets $A_0\subset A$ and
$B_0\subset B$ supporting probability measures with exponents
$\alpha$ and $\beta$. Apply Theorem~\ref{int:intersection-theorem} to
these measures and use
$gA_0\cap(B_0+z)\subset gA\cap(B+z)$. This proves \textup{(ii)} and
completes the proof.
\end{proof}

\begin{remark}\label{int:intersection-equality}
Under the hypotheses of Corollary~\ref{int:set-intersections}\textup{(ii)},
suppose in addition that $\dimH(A\times B)=\alpha+\beta$.
Then the conclusion holds with equality:
\[
 \dimH\bigl(gA\cap(B+z)\bigr)=\alpha+\beta-2
\]
for almost all $g$ and a set of translations of positive Lebesgue measure.
Indeed, for every fixed $g$, the standard upper bound for fibres of the
Lipschitz map $(a,b)\mapsto ga-b$ gives
\[
 \dimH\bigl(gA\cap(B+z)\bigr)
 \le\dimH(A\times B)-2
 \qquad\text{for Lebesgue almost all }z;
\]
see \cite[Theorem~7.7]{Mattila}. Removing this null set of translations from $Z_g$
proves the assertion. The product-dimension hypothesis holds, in
particular, when
$\dimH A=\dimP A=\alpha$ and $\dimH B=\dimP B=\beta$, by the standard
product inequalities.
\end{remark}

\section*{Acknowledgements}

The authors would like to thank Professor Pertti Mattila for many helpful discussions and valuable suggestions concerning this work. D. T. Tran would like to thank the Vietnam Institute for Advanced Study in Mathematics (VIASM) for its warm hospitality and excellent research environment


\begin{thebibliography}{GGGHMW}
\bibitem[CLPY]{CLPY}
A. Carbery, Z. K. Li, Y. Pang and P.-L. Yung.
A weighted formulation of refined decoupling and inequalities of
Mizohata--Takeuchi-type for the moment curve,
\emph{J. Geom. Anal.} \textbf{36} (2026), Article~48.

\bibitem[Cr]{Craig}
S. Craig.
The quasi-Assouad dimension of $(1,2t)$-Furstenberg sets in $\R^3$ is
extremized by sticky sets, arXiv:2510.06462v2 (2025).

\bibitem[FF]{FF}
L. R. Ford, Jr. and D. R. Fulkerson.
Maximal flow through a network,
\emph{Canad. J. Math.} \textbf{8} (1956), 399--404.

\bibitem[FO]{FO}
K. F\"assler and T. Orponen.
On restricted families of projections in \(\R^3\),
\emph{Proc. Lond. Math. Soc.} \textbf{109} (2014), no.~2, 353--381.

\bibitem[GGGHMW]{GGGHMW}
S. Gan, S. Guo, L. Guth, T. L. J. Harris, D. Maldague and H. Wang.
On restricted projections to planes in \(\R^3\),
\emph{Amer. J. Math.} \textbf{148} (2026), no.~4, 915--960.

\bibitem[GGM]{GGM}
S. Gan, L. Guth and D. Maldague.
An exceptional set estimate for restricted projections to lines in \(\R^3\),
\emph{J. Geom. Anal.} \textbf{34} (2024), Article~15.

\bibitem[GIOW]{GIOW}
L. Guth, A. Iosevich, Y. Ou and H. Wang.
On Falconer's distance set problem in the plane,
\emph{Invent. Math.} \textbf{219} (2020), no.~3, 779--830.

\bibitem[GWZ]{GWZ}
L. Guth, H. Wang and J. Zahl.
A streamlined proof of the Kakeya set conjecture in \(\R^3\),
arXiv:2601.14411v1 (2026).

\bibitem[KOV]{KOV}
A. K\"aenm\"aki, T. Orponen and L. Venieri.
A Marstrand-type restricted projection theorem in \(\R^3\),
\emph{Amer. J. Math.} \textbf{147} (2025), no.~1, 81--123.

\bibitem[M1]{Mattila1984}
P. Mattila.
Hausdorff dimension and capacities of intersections of sets in \(n\)-space,
\emph{Acta Math.} \textbf{152} (1984), 77--105.

\bibitem[M2]{Mattila1985}
P. Mattila.
On the Hausdorff dimension and capacities of intersections,
\emph{Mathematika} \textbf{32} (1985), no.~2, 213--217.

\bibitem[M3]{Mattila}
P. Mattila.
\emph{Geometry of Sets and Measures in Euclidean Spaces:
Fractals and Rectifiability}.
Cambridge Studies in Advanced Mathematics 44,
Cambridge University Press, 1995.

\bibitem[M4]{Mattila2022}
P. Mattila.
Hausdorff dimension and projections related to intersections,
\emph{Publ. Mat.} \textbf{66} (2022), no.~1, 305--323.

\bibitem[M5]{MattilaSurvey}
P. Mattila.
A survey on the Hausdorff dimension of intersections,
\emph{Math. Comput. Appl.} \textbf{28} (2023), no.~2, Article~49.

\bibitem[M6]{Mattila2024}
P. Mattila.
Hausdorff dimension of plane sections and general intersections,
\emph{Bull. London Math. Soc.} \textbf{56} (2024), no.~6, 1988--1998.

\bibitem[M7]{Mattila2017}
P. Mattila.
Hausdorff dimension, projections, intersections, and Besicovitch sets,
in \emph{New Trends in Applied Harmonic Analysis, Volume~2},
A. Aldroubi, C. Cabrelli, S. Jaffard and U. Molter (eds.),
Applied and Numerical Harmonic Analysis,
Birkh\"auser, Cham, 2019, 129--157.

\bibitem[OWY]{OWY}
W. O'Regan, S. Wu and G. Yi.
Two-ends Furstenberg inequality for transversal families and applications
to Fourier decay, arXiv:2607.08461 (2026).

\bibitem[OS]{OS}
T. Orponen and P. Shmerkin.
Projections, Furstenberg sets, and the \(ABC\) sum-product problem,
\emph{J. Amer. Math. Soc.} \textbf{39} (2026), no.~3, 857--913.

\bibitem[PY]{PY}
T. Pham and S. Yoo.
Intersection patterns and connections to distance problems,
arXiv:2304.08004v7 (2025).

\bibitem[PYZ]{PYZ}
M. Pramanik, T. Yang and J. Zahl.
A Furstenberg-type problem for circles, and a Kaufman-type restricted
projection theorem in \(\R^3\),
\emph{Amer. J. Math.}, to appear; arXiv:2207.02259v3.

\bibitem[RW]{RW}
K. Ren and H. Wang.
Furstenberg sets estimate in the plane, arXiv:2308.08819v3 (2025).

\bibitem[WW1]{WW2024}
H. Wang and S. Wu.
Restriction estimates using decoupling theorems and two-ends Furstenberg
inequalities, arXiv:2411.08871v3 (2024).

\bibitem[WW2]{WW}
H. Wang and S. Wu.
Two-ends Furstenberg estimates in the plane,
\emph{Math. Ann.} \textbf{395} (2026), Article~91.

\bibitem[WZ]{WZ}
H. Wang and J. Zahl.
Volume estimates for unions of convex sets, and the Kakeya set
conjecture in three dimensions, arXiv:2502.17655v1 (2025).
\bibitem[Yu]{Yu}
H. Yu.
On GILP's group-theoretic approach to Falconer's distance problem,
\emph{Glasgow Math. J.} \textbf{63} (2021), no.~3, 547--562.

\end{thebibliography}
\end{document}